\documentclass{amsart}
\usepackage{a4}
\usepackage{hyperref}

\usepackage{comment}

\usepackage{xcolor}

\usepackage{amsmath}
\usepackage{amssymb}
\usepackage[backrefs,lite]{amsrefs}

\usepackage{longtable}
\usepackage{booktabs}

\usepackage{enumerate}
\usepackage{enumitem}

\usepackage[all]{xy}
\usepackage{tikz}		
\usepackage{tikz-3dplot}
\usepackage{tikz-cd}
\usepackage{graphicx}

\usetikzlibrary{patterns,matrix,arrows,shapes}

\theoremstyle{plain}

\newtheorem{theorem}{Theorem}[section]
\newtheorem{lemma}[theorem]{Lemma}
\newtheorem{proposition}[theorem]{Proposition}
\newtheorem{corollary}[theorem]{Corollary}

\theoremstyle{definition}

\newtheorem{definition}[theorem]{Definition}
\newtheorem{example}[theorem]{Example}

\newtheorem{construction}[theorem]{Construction}

\newtheorem{remark}[theorem]{Remark}

\theoremstyle{remark}

\numberwithin{equation}{section}

\newcommand\CC{{\mathbb C}}

\newcommand\TT{{\mathbb T}}
\newcommand\ZZ{{\mathbb Z}}

\newcommand\QQ{{\mathbb Q}}
\newcommand\PP{{\mathbb P}}

\newcommand\KKK{{\mathcal K}}
\newcommand\RRR{{\mathcal R}}
\newcommand\OOO{{\mathcal O}}

\newcommand\Eff{{\rm Eff}}
\newcommand\Mov{{\rm Mov}}
\newcommand\Ample{{\rm Ample}}
\newcommand\SAmple{{\rm SAmple}}

\newcommand\trop{{\rm trop}}
\renewcommand\div{{\rm div}}
\newcommand\Cl{\operatorname{Cl}}

\newcommand\conv{{\rm conv}}
\newcommand\cone{{\rm cone}}

\newcommand\Spec{{\rm Spec}}

\newcommand\quot{/\!\!/}
\newcommand\tm{{\tau^-}}
\newcommand\tp{{\tau^+}}
\newcommand\tx{{\tau_X}}

\newcommand\lin{{\rm lin}}

\newcommand\Ker{{\rm Ker}\,}
\newcommand\rk{{\rm rk}\,}

\newcommand\im{{\rm im}}
\newcommand\V{{\rm V}}

\newcommand\bangle[1]{\langle #1 \rangle}

\newcommand\crit{{\mathrm{crit}}}

\begin{document}
\title[On Fano threefolds with $\CC^*$-action]%
{On Fano threefolds with $\CC^*$-action}


\author[C.~Hische]{Christoff Hische} 
\address{Mathematisches Institut, Eberhard Karls Universit\"at T\"ubingen, Auf der Morgenstelle 10, 72076 T\"ubingen, Germany}
\email{hische@math.uni-tuebingen.de}

\author[M.~Wrobel]{Milena Wrobel} 
\address{Institut f\"ur Mathematik, Carl von Ossietzky Universit\"at Oldenburg, 26111 Oldenburg, Germany}
\email{milena.wrobel@uni-oldenburg.de}

\begin{abstract}
We study three-dimensional Fano varieties with $\CC^*$-action.
Complementing recent results~\cite{HiWr2018}, we give classification results in the canonical case, where the maximal orbit quotient is $\PP_2$ having  a line arrangement of five lines in special position as its critical values.
\end{abstract}

\maketitle

\section{Introduction}
In this paper we study complex Fano varieties $X$ coming
with an effective torus action $T \times X \to X$. 
Recall that a Fano variety is a normal projective algebraic 
variety with an ample anticanonical divisor. 
The general classification of smooth Fano varieties
is complete up to dimension three~\cites{Is1977,Is1978,MoMu1981} in the smooth
case and stays widely open in higher dimensions or
in the singular case. 
The presence of a torus action makes the situation
more accessible.
For instance, in the important example class of toric
Fano varieties, the smooth ones are known up to dimension
nine~\cites{Ba1981,WaWa1982,Ba1999,Sa2000,Ob2007,KrNi2009,Pa} and in the case of at most canonical
singularities, there is a complete classification
up to dimension three~\cites{Ka2006,Ka2010}.

Leaving the toric case means that the complexity
of the torus action, i.e.~the codimension of the
general torus orbit, becomes non-zero.
Then the {\em maximal orbit quotient}
comes into the game.
This is a rational map $\pi \colon X \dashrightarrow Y$,
which satisfies $\CC(X)^T = \CC(Y)$ and is 
defined on a large open set $W \subseteq X$
of points with finite isotropy groups;
see Section~\ref{sec:2} for the details. 
The first non-trivial step is complexity one.
Here, the target variety $Y$ of the maximal
orbit quotient turns out to be the projective
line~$\PP_1$.
In this setting, the smooth Fano varieties of Picard
number at most two are known in any dimension~\cite{FaHaNi2018}
and the terminal threefolds of Picard number one 
have been classified~\cites{BeHaHuNi2016}.

The next step, entering complexity two, is a major one,
comparable to passing from curves to surfaces.
We focus here on the case $Y = \PP_2$.
The crucial data is $\crit(\pi)$, the collection
of divisors on $Y$ consisting of critical
values of~$\pi$.
The simplest situation are line configurations in
$Y = \PP_2$.
The case that all lines are in general position
means that $X$ is one of
the~\emph{general arrangement varieties}
introduced in~\cite{HaHiWr2019}.
In this setting, we have a complete classification
of the smooth Fanos of complexity two
and Picard number at most two~\cite{HaHiWr2019}
and a first partial classification in for
threefolds of Picard number one with at most
canonical singularities~\cite{HiWr2018}.

In the present paper, we enter the case of
special configurations, in other words,
the case of~\emph{special arrangement varieties}.
More precisely, we consider {\em honestly special arrangement varieties}, meaning that they do not admit a torus action turning them into a general arrangement variety.
Our Fano varieties $X$ will be of Picard number
at most two with $Y = \PP_2$.
In this situation we have at least five lines in $\crit(\pi)$; see Proposition~\ref{prop:fiveLines}.
As we will see in Theorem~\ref{thm:noSmoothOnes}, such varieties
are never smooth.
Turning to the singular case, when we say that $X$ is of
\emph{finite isotropy order at most $k$},
we mean that there is an open set
$U \subseteq X$ with complement $X \setminus U$
of codimension at least two such that the
isotropy group $T_x$ is either infinite or
of order at most $k$ for all $x \in U$.

\begin{theorem}\label{thm:classification}
Every canonical Fano honestly special arrangement
threefold of complexiy two, Picard number at most two,
five critical lines and finite isotropy order at most
two is isomorphic to one of the following Fano
threefolds $X$, specified by
its $\mathrm{Cl}(X)$-graded Cox ring $\mathcal{R}(X)$,
the matrix $Q=[w_1,\ldots,w_r]$ of generator 
degrees and the anticanonical class
$-\mathcal{K}_X \in \Cl(X)$.
{\small \setlength{\arraycolsep}{3pt}
\renewcommand*{\arraystretch}{1.1}
\begin{longtable}{c|c|c|c|c}
No.&
$\RRR(X)$&
$\Cl(X)$&
$\left[w_1,\ldots,w_r\right]$&
$u$
\\
\hline
1&
{\tiny
$
\begin{array}{c}
\CC[T_{1},T_{2},T_{3},T_{4},T_{5},S_{1}]
 \\ 
 \hline
\left\langle\begin{array}{c}
T_{1}^2+T_{2}^2+T_{3}^2+T_{4}^2,
\\
T_{2}^2+aT_{3}^2+T_{5}^2
\end{array}\right\rangle
\\
a\not =0, 1
\end{array}
$
}
&
$\ZZ\times(\ZZ_2)^4$
&
{\tiny
$\left[\begin{array}{cccccc}
1&1&1&1&1&1
\\
\bar{1}&\bar{1}&\bar{1}&\bar{1}&\bar{0}&\bar{0}
\\
\bar{1}&\bar{1}&\bar{1}&\bar{0}&\bar{1}&\bar{0}
\\
\bar{1}&\bar{1}&\bar{0}&\bar{0}&\bar{0}&\bar{0}
\\
\bar{1}&\bar{0}&\bar{1}&\bar{1}&\bar{1}&\bar{0}
\end{array}\right]$
}
&
{\tiny
$
\left[
\begin{array}{c}
2\\
\bar{0}\\
\bar{0}\\
\bar{0}\\
\bar{0}\\
\end{array}
\right]
$
}
\\
\hline
2&
{\tiny
$
\begin{array}{c}
\CC[T_{1},T_{2},T_{3},T_{4},T_{5},S_{1}]
 \\ 
 \hline
\left\langle\begin{array}{c}
T_{1}^2+T_{2}^2+T_{4}^2,
\\
T_{1}^2+T_{3}^2+T_{5}^2
\end{array}\right\rangle
\end{array}
$
}
&
$\ZZ\times(\ZZ_2)^4$&
{\tiny
$
\left[\begin{array}{cccccc}
1&1&1&1&1&1
\\
\bar{1}&\bar{1}&\bar{1}&\bar{1}&\bar{0}&\bar{0}
\\
\bar{1}&\bar{1}&\bar{1}&\bar{0}&\bar{1}&\bar{0}
\\
\bar{1}&\bar{1}&\bar{0}&\bar{0}&\bar{0}&\bar{0}
\\
\bar{1}&\bar{0}&\bar{1}&\bar{1}&\bar{1}&\bar{0}
\end{array}\right]
$
}&
{\tiny
$
\left[
\begin{array}{c}
2
\\
\bar 0
\\
\bar 0
\\
\bar 0
\\
\bar 0
\end{array}
\right]
$
}
\\
\hline
3&
{\tiny
$
\begin{array}{c}
\CC[T_{1},T_{2},T_{3},T_{4},T_{5},T_{6}]
 \\ 
 \hline
\left\langle\begin{array}{c}
T_{1}T_{2}+T_{3}^2+T_{4}^2+T_{5}^2,
\\
T_{3}^2+aT_{4}^2+T_{6}^2
\end{array}\right\rangle
\\
a\not =0, 1
\end{array}
$
}
&
$\ZZ\times(\ZZ_2)^2\times\ZZ_4$&
{\tiny
$
\left[\begin{array}{cccccc}
1&1&1&1&1&1
\\
\bar{1}&\bar{1}&\bar{1}&\bar{0}&\bar{1}&\bar{0}
\\
\bar{0}&\bar{0}&\bar{0}&\bar{1}&\bar{1}&\bar{0}
\\
\bar{1}&\bar{3}&\bar{2}&\bar{0}&\bar{0}&\bar{0}
\end{array}\right]
$
}&
{\tiny
$
\left[
\begin{array}{c}
2
\\
\bar 0
\\
\bar 0
\\
\bar 2
\end{array}
\right]
$
}
\\
\hline
4&
{\tiny
$
\begin{array}{c}
\CC[T_{1},T_{2},T_{3},T_{4},T_{5},T_{6}]
 \\ 
 \hline
\left\langle\begin{array}{c}
T_{1}T_{2}+T_{3}^2+T_{5}^2,
\\
T_{1}T_{2}+T_{4}^2+T_{6}^2
\end{array}\right\rangle
\end{array}
$
}
&
$\ZZ\times(\ZZ_2)^2\times\ZZ_4$&
{\tiny
$
\left[\begin{array}{cccccc}
1&1&1&1&1&1
\\
\bar{1}&\bar{1}&\bar{1}&\bar{0}&\bar{1}&\bar{0}
\\
\bar{0}&\bar{0}&\bar{0}&\bar{1}&\bar{1}&\bar{0}
\\
\bar{3}&\bar{1}&\bar{2}&\bar{0}&\bar{0}&\bar{0}
\end{array}\right]
$
}&
{\tiny
$
\left[
\begin{array}{c}
2
\\
\bar 0
\\
\bar 0
\\
\bar 2
\end{array}
\right]
$
}
\\
\hline
5&
{\tiny
$
\begin{array}{c}
\CC[T_{1},T_{2},T_{3},T_{4},T_{5},T_{6}]
 \\ 
 \hline
\left\langle\begin{array}{c}
T_{1}^2+T_{2}T_{3}+T_{5}^2,
\\
T_{1}^2+T_{4}^2+T_{6}^2
\end{array}\right\rangle
\end{array}
$
}
&
$\ZZ\times(\ZZ_2)^2\times\ZZ_4$&
{\tiny
$
\left[\begin{array}{cccccc}
1&1&1&1&1&1
\\
\bar{1}&\bar{1}&\bar{1}&\bar{0}&\bar{1}&\bar{0}
\\
\bar{0}&\bar{0}&\bar{0}&\bar{1}&\bar{1}&\bar{0}
\\
\bar{2}&\bar{1}&\bar{3}&\bar{0}&\bar{0}&\bar{0}
\end{array}\right]
$
}&
{\tiny
$
\left[
\begin{array}{c}
2
\\
\bar 0
\\
\bar 0
\\
\bar 2
\end{array}
\right]
$
}
\\
\hline
6&
{\tiny
$
\begin{array}{c}
\CC[T_{1},T_{2},T_{3},T_{4},T_{5},T_{6}]
 \\ 
 \hline
\left\langle\begin{array}{c}
T_{1}^2+T_{2}T_{3}+T_{4}^2+T_{5}^2,
\\
T_{2}T_{3}+aT_{4}^2+T_{6}^2
\end{array}\right\rangle
\\
a\not =0, 1
\end{array}
$
}
&
$\ZZ\times(\ZZ_2)^2\times\ZZ_4$&
{\tiny
$
\left[\begin{array}{cccccc}
1&1&1&1&1&1
\\
\bar{1}&\bar{1}&\bar{1}&\bar{0}&\bar{1}&\bar{0}
\\
\bar{0}&\bar{0}&\bar{0}&\bar{1}&\bar{1}&\bar{0}
\\
\bar{2}&\bar{3}&\bar{1}&\bar{0}&\bar{0}&\bar{0}
\end{array}\right]
$
}&
{\tiny
$
\left[
\begin{array}{c}
2
\\
\bar0
\\
\bar0
\\
\bar2
\end{array}
\right]
$
}
\\
\hline
7&
{\tiny
$
\begin{array}{c}
\CC[T_{1},T_{2},T_{3},T_{4},T_{5},T_{6},T_{7}]
 \\ 
 \hline
\left\langle\begin{array}{c}
T_{1}T_{2}+T_{3}^2+T_{4}^2+T_{5}T_{6},
\\
T_{3}^2+aT_{4}^2+T_{7}^2
\end{array}\right\rangle
\\
a\not =0, 1
\end{array}
$
}
&
$\ZZ^2\times(\ZZ_2)^2$&
{\tiny
$
\left[\begin{array}{ccccccc}
1&-1&0&0&-1&1&0
\\
2&0&1&1&1&1&1
\\
\bar{0}&\bar{0}&\bar{1}&\bar{1}&\bar{0}&\bar{0}&\bar{0}
\\
\bar{1}&\bar{1}&\bar{0}&\bar{1}&\bar{0}&\bar{0}&\bar{0}
\end{array}\right]
$
}&
{\tiny
$
\left[
\begin{array}{c}
0
\\
3
\\
\bar{0}
\\
\bar{1}
\end{array}
\right]
$
}
\\
\hline
8&
{\tiny
$
\begin{array}{c}
\CC[T_{1},T_{2},T_{3},T_{4},T_{5},T_{6},T_{7}]
 \\ 
 \hline
\left\langle\begin{array}{c}
T_{1}T_{2}+T_{3}^2+T_{4}^2+T_{5}T_{6},
\\
T_{3}^2+aT_{4}^2+T_{7}^2
\end{array}\right\rangle
\\
a\not =0, 1
\end{array}
$
}
&
$\ZZ^2\times(\ZZ_2)^2$&
{\tiny
$
\left[\begin{array}{ccccccc}
2&-2&0&0&-1&1&0
\\
1&1&1&1&1&1&1
\\
\bar{0}&\bar{0}&\bar{1}&\bar{1}&\bar{0}&\bar{0}&\bar{0}
\\
\bar{1}&\bar{1}&\bar{0}&\bar{1}&\bar{0}&\bar{0}&\bar{0}
\end{array}\right]
$
}&
{\tiny
$
\left[
\begin{array}{c}
0
\\
3
\\
\bar{0}
\\
\bar{1}
\end{array}
\right]
$
}
\\
\hline
9&
{\tiny
$
\begin{array}{c}
\CC[T_{1},T_{2},T_{3},T_{4},T_{5},T_{6},T_{7}]
 \\ 
 \hline
\left\langle\begin{array}{c}
T_{1}^2+T_{2}T_{3}+T_{4}T_{5}+T_{6}^2,
\\
T_{2}T_{3}+aT_{4}T_{5}+T_{7}^2\\
a\not =0, 1
\end{array}\right\rangle
\end{array}
$
}
&
$\ZZ^2\times(\ZZ_2)^2$&
{\tiny
$
\left[\begin{array}{ccccccc}
0&-1&1&1&-1&0&0
\\
1&1&1&2&0&1&1
\\
\bar{1}&\bar{0}&\bar{0}&\bar{0}&\bar{0}&\bar{1}&\bar{0}
\\
\bar{0}&\bar{0}&\bar{0}&\bar{1}&\bar{1}&\bar{1}&\bar{0}
\end{array}\right]
$
}&
{\tiny
$
\left[
\begin{array}{c}
0
\\
3
\\
\bar{0}
\\
\bar{1}
\end{array}
\right]
$
}
\\
\hline
10&
{\tiny
$
\begin{array}{c}
\CC[T_{1},T_{2},T_{3},T_{4},T_{5},T_{6},T_{7}]
 \\ 
 \hline
\left\langle\begin{array}{c}
T_{1}^2+T_{2}T_{3}+T_{4}T_{5}+T_{6}^2,
\\
T_{2}T_{3}+aT_{4}T_{5}+T_{7}^2
\\
a\not =0, 1
\end{array}\right\rangle
\end{array}
$
}
&
$\ZZ^2\times(\ZZ_2)^2$&
{\tiny
$
\left[\begin{array}{ccccccc}
0&2&-2&-1&1&0&0
\\
1&1&1&1&1&1&1
\\
\bar{1}&\bar{0}&\bar{0}&\bar{0}&\bar{0}&\bar{1}&\bar{0}
\\
\bar{0}&\bar{1}&\bar{1}&\bar{0}&\bar{0}&\bar{1}&\bar{0}
\end{array}\right]
$
}&
{\tiny
$
\left[
\begin{array}{c}
0
\\
3
\\
\bar{0}
\\
\bar{1}
\end{array}
\right]
$
}
\\
\hline
11&
{\tiny
$
\begin{array}{c}
\CC[T_{1},T_{2},T_{3},T_{4},T_{5},T_{6},T_{7}]
 \\ 
 \hline
\left\langle\begin{array}{c}
T_{1}T_{2}+T_{3}T_{4}+T_{6}^2,
\\
T_{1}T_{2}+T_{5}^2+T_{7}^2
\end{array}\right\rangle
\end{array}
$
}
&
$\ZZ^2\times(\ZZ_2)^2$&
{\tiny
$
\left[\begin{array}{ccccccc}
1&-1&-1&1&0&0&0
\\
2&0&1&1&1&1&1
\\
\bar{0}&\bar{0}&\bar{0}&\bar{0}&\bar{1}&\bar{1}&\bar{0}
\\
\bar{1}&\bar{1}&\bar{0}&\bar{0}&\bar{0}&\bar{1}&\bar{0}
\end{array}\right]
$
}&
{\tiny
$
\left[
\begin{array}{c}
0
\\
3
\\
\bar 0
\\
\bar 1
\end{array}
\right]
$
}
\\
\hline
12&
{\tiny
$
\begin{array}{c}
\CC[T_{1},T_{2},T_{3},T_{4},T_{5},T_{6},T_{7}]
 \\ 
 \hline
\left\langle\begin{array}{c}
T_{1}T_{2}+T_{3}T_{4}+T_{6}^2,
\\
T_{1}T_{2}+T_{5}^2+T_{7}^2
\end{array}\right\rangle
\end{array}
$
}
&
$\ZZ^2\times(\ZZ_2)^2$&
{\tiny
$
\left[\begin{array}{ccccccc}
-1&1&0&0&0&0&0
\\
1&1&1&1&1&1&1
\\
\bar{0}&\bar{0}&\bar{0}&\bar{0}&\bar{1}&\bar{1}&\bar{0}
\\
\bar{0}&\bar{0}&\bar{1}&\bar{1}&\bar{0}&\bar{1}&\bar{0}
\end{array}\right]
$
}&
{\tiny
$
\left[
\begin{array}{c}
0
\\
3
\\
\bar 0
\\
\bar 1
\end{array}
\right]
$
}
\\
\hline
13&
{\tiny
$
\begin{array}{c}
\CC[T_{1},T_{2},T_{3},T_{4},T_{5},T_{6},T_{7}]
 \\ 
 \hline
\left\langle\begin{array}{c}
T_{1}T_{2}+T_{3}T_{4}+T_{6}^2,
\\
T_{1}T_{2}+T_{5}^2+T_{7}^2
\end{array}\right\rangle
\end{array}
$
}
&
$\ZZ^2\times(\ZZ_2)^2$&
{\tiny
$
\left[\begin{array}{ccccccc}
-1&1&0&0&0&0&0
\\
2&2&1&3&2&2&2
\\
\bar{0}&\bar{0}&\bar{1}&\bar{1}&\bar{1}&\bar{1}&\bar{0}
\\
\bar{0}&\bar{0}&\bar{0}&\bar{0}&\bar{1}&\bar{0}&\bar{1}
\end{array}\right]
$
}&
{\tiny
$
\left[
\begin{array}{c}
0
\\
6
\\
\bar 0
\\
\bar 1
\end{array}
\right]
$
}
\\
\hline
14&
{\tiny
$
\begin{array}{c}
\CC[T_{1},T_{2},T_{3},T_{4},T_{5},T_{6},T_{7}]
 \\ 
 \hline
\left\langle\begin{array}{c}
T_{1}^2+T_{2}T_{3}+T_{6}^2,
\\
T_{1}^2+T_{4}T_{5}+T_{7}^2
\end{array}\right\rangle
\end{array}
$
}
&
$\ZZ^2\times(\ZZ_2)^2$&
{\tiny
$
\left[\begin{array}{ccccccc}
0&-1&1&1&-1&0&0
\\
1&1&1&2&0&1&1
\\
\bar{1}&\bar{0}&\bar{0}&\bar{0}&\bar{0}&\bar{1}&\bar{0}
\\
\bar{0}&\bar{0}&\bar{0}&\bar{1}&\bar{1}&\bar{1}&\bar{0}
\end{array}\right]
$
}&
{\tiny
$
\left[
\begin{array}{c}
0
\\
3
\\
\bar 0
\\
\bar 1
\end{array}
\right]
$
}
\\
\hline
15&
{\tiny
$
\begin{array}{c}
\CC[T_{1},T_{2},T_{3},T_{4},T_{5},T_{6},T_{7}]
 \\ 
 \hline
\left\langle\begin{array}{c}
T_{1}^2+T_{2}T_{3}+T_{6}^2,
\\
T_{1}^2+T_{4}T_{5}+T_{7}^2
\end{array}\right\rangle
\end{array}
$
}
&
$\ZZ^2\times(\ZZ_2)^2$&
{\tiny
$
\left[\begin{array}{ccccccc}
0&-1&1&0&0&0&0
\\
1&1&1&1&1&1&1
\\
\bar{1}&\bar{0}&\bar{0}&\bar{0}&\bar{0}&\bar{1}&\bar{0}
\\
\bar{0}&\bar{0}&\bar{0}&\bar{1}&\bar{1}&\bar{1}&\bar{0}
\end{array}\right]
$
}&
{\tiny
$
\left[
\begin{array}{c}
0
\\
3
\\
\bar 0
\\
\bar 1
\end{array}
\right]
$
}
\\
\hline
16&
{\tiny
$
\begin{array}{c}
\CC[T_{1},T_{2},T_{3},T_{4},T_{5},T_{6},T_{7}]
 \\ 
 \hline
\left\langle\begin{array}{c}
T_{1}^2+T_{2}T_{3}+T_{6}^2,
\\
T_{1}^2+T_{4}T_{5}+T_{7}^2
\end{array}\right\rangle
\end{array}
$
}
&
$\ZZ^2\times(\ZZ_2)^2$&
{\tiny
$
\left[\begin{array}{ccccccc}
0&-1&1&2&-2&0&0
\\
1&1&1&1&1&1&1
\\
\bar{1}&\bar{0}&\bar{0}&\bar{0}&\bar{0}&\bar{1}&\bar{0}
\\
\bar{0}&\bar{0}&\bar{0}&\bar{1}&\bar{1}&\bar{1}&\bar{0}
\end{array}\right]
$
}&
{\tiny
$
\left[
\begin{array}{c}
0
\\
3
\\
\bar 0
\\
\bar 1
\end{array}
\right]
$
}
\\
\hline
17&
{\tiny
$
\begin{array}{c}
\CC[T_{1},T_{2},T_{3},T_{4},T_{5},T_{6},T_{7}]
 \\ 
 \hline
\left\langle\begin{array}{c}
T_{1}^2+T_{2}T_{3}+T_{6}^2,
\\
T_{1}^2+T_{4}T_{5}+T_{7}^2
\end{array}\right\rangle
\end{array}
$
}
&
$\ZZ^2\times(\ZZ_2)^2$&
{\tiny
$
\left[\begin{array}{ccccccc}
0&-1&1&0&0&0&0
\\
2&2&2&1&3&2&2
\\
\bar{1}&\bar{0}&\bar{0}&\bar{1}&\bar{1}&\bar{1}&\bar{0}
\\
\bar{1}&\bar{0}&\bar{0}&\bar{0}&\bar{0}&\bar{0}&\bar{1}
\end{array}\right]
$
}&
{\tiny
$
\left[
\begin{array}{c}
0
\\
6
\\
\bar 0
\\
\bar 1
\end{array}
\right]
$
}
\\
\hline
18&
{\tiny
$
\begin{array}{c}
\CC[T_{1},T_{2},T_{3},T_{4},T_{5},T_{6},T_{7}]
 \\ 
 \hline
\left\langle\begin{array}{c}
T_{1}T_{2}+T_{3}T_{4}+T_{5}^2+T_{6}^2,
\\
T_{3}T_{4}+aT_{5}^2+T_{7}^2
\end{array}\right\rangle
\\
a\not =0, 1
\end{array}
$
}
&
$\ZZ^2\times(\ZZ_2)^2$&
{\tiny
$
\left[\begin{array}{ccccccc}
1&-1&-1&1&0&0&0
\\
2&0&1&1&1&1&1
\\
\bar{0}&\bar{0}&\bar{0}&\bar{0}&\bar{1}&\bar{1}&\bar{0}
\\
\bar{1}&\bar{1}&\bar{0}&\bar{0}&\bar{0}&\bar{1}&\bar{0}
\end{array}\right]
$
}&
{\tiny
$
\left[
\begin{array}{c}
0
\\
3
\\
\bar{0}
\\
\bar{1}
\end{array}
\right]
$
}
\\
\hline
19&
{\tiny
$
\begin{array}{c}
\CC[T_{1},T_{2},T_{3},T_{4},T_{5},T_{6},T_{7}]
 \\ 
 \hline
\left\langle\begin{array}{c}
T_{1}^2+T_{2}^2+T_{3}^2+T_{4}T_{5},
\\
T_{2}^2+aT_{3}^2+T_{6}T_{7}
\end{array}\right\rangle
\\
a\not =0, 1
\end{array}
$
}
&
$\ZZ^2\times(\ZZ_2)^2$&
{\tiny
$
\left[\begin{array}{ccccccc}
0&0&0&1&-1&-1&1
\\
1&1&1&2&0&1&1
\\
\bar{1}&\bar{1}&\bar{0}&\bar{0}&\bar{0}&\bar{0}&\bar{0}
\\
\bar{0}&\bar{1}&\bar{0}&\bar{1}&\bar{1}&\bar{0}&\bar{0}
\end{array}\right]
$
}&
{\tiny
$
\left[
\begin{array}{c}
0
\\
3
\\
\bar{0}
\\
\bar{1}
\end{array}
\right]
$
}
\\
\hline
20&
{\tiny
$
\begin{array}{c}
\CC[T_{1},T_{2},T_{3},T_{4},T_{5},T_{6},T_{7}]
 \\ 
 \hline
\left\langle\begin{array}{c}
T_{1}T_{2}+T_{3}T_{4}+T_{5}^2+T_{6}^2,
\\
T_{3}T_{4}+aT_{5}^2+T_{7}^2
\end{array}\right\rangle
\\
a\not =0, 1
\end{array}
$
}
&
$\ZZ^2\times(\ZZ_2)^2$&
{\tiny
$
\left[\begin{array}{ccccccc}
2&-2&-1&1&0&0&0
\\
1&1&1&1&1&1&1
\\
\bar{0}&\bar{0}&\bar{0}&\bar{0}&\bar{1}&\bar{1}&\bar{0}
\\
\bar{1}&\bar{1}&\bar{0}&\bar{0}&\bar{0}&\bar{1}&\bar{0}
\end{array}\right]
$
}&
{\tiny
$
\left[
\begin{array}{c}
0
\\
3
\\
\bar{0}
\\
\bar{1}
\end{array}
\right]
$
}
\\
\hline
21&
{\tiny
$
\begin{array}{c}
\CC[T_{1},T_{2},T_{3},T_{4},T_{5},T_{6},T_{7}]
 \\ 
 \hline
\left\langle\begin{array}{c}
T_{1}T_{2}+T_{3}T_{4}+T_{5}^2+T_{6}^2,
\\
T_{3}T_{4}+aT_{5}^2+T_{7}^2
\end{array}\right\rangle
\\
a\not =0, 1
\end{array}
$
}
&
$\ZZ^2\times(\ZZ_2)^2$&
{\tiny
$
\left[\begin{array}{ccccccc}
-1&1&2&-2&0&0&0
\\
1&1&1&1&1&1&1
\\
\bar{0}&\bar{0}&\bar{0}&\bar{0}&\bar{1}&\bar{1}&\bar{0}
\\
\bar{0}&\bar{0}&\bar{1}&\bar{1}&\bar{0}&\bar{1}&\bar{0}
\end{array}\right]
$
}&
{\tiny
$
\left[
\begin{array}{c}
0
\\
3
\\
\bar{0}
\\
\bar{1}
\end{array}
\right]
$
}
\\
\hline
22&
{\tiny
$
\begin{array}{c}
\CC[T_{1},T_{2},T_{3},T_{4},T_{5},T_{6},T_{7}]
 \\ 
 \hline
\left\langle\begin{array}{c}
T_{1}^2+T_{2}^2+T_{3}^2+T_{4}T_{5},
\\
T_{2}^2+aT_{3}^2+T_{6}T_{7}
\end{array}\right\rangle
\\
a\not =0, 1
\end{array}
$
}
&
$\ZZ^2\times(\ZZ_2)^2$&
{\tiny
$
\left[\begin{array}{ccccccc}
1&1&1&1&1&1&1
\\
0&0&0&-2&2&1&-1
\\
\bar{1}&\bar{1}&\bar{0}&\bar{0}&\bar{0}&\bar{0}&\bar{0}
\\
\bar{0}&\bar{1}&\bar{0}&\bar{1}&\bar{1}&\bar{0}&\bar{0}
\end{array}\right]
$
}&
{\tiny
$
\left[
\begin{array}{c}
3
\\
0
\\
\bar{0}
\\
\bar{1}
\end{array}
\right]
$
}
\\
\hline
23&
{\tiny
$
\begin{array}{c}
\CC[T_{1},T_{2},T_{3},T_{4},T_{5},T_{6},T_{7}]
 \\ 
 \hline
\left\langle\begin{array}{c}
T_{1}^2+T_{2}^2+T_{3}^2+T_{4}T_{5},
\\
T_{2}^2+aT_{3}^2+T_{6}T_{7}
\end{array}\right\rangle
\\
a\not =0, 1
\end{array}
$
}
&
$\ZZ^2\times(\ZZ_2)^2$&
{\tiny
$
\left[\begin{array}{ccccccc}
1&1&1&1&1&1&1
\\
0&0&0&1&-1&-2&2
\\
\bar{1}&\bar{1}&\bar{0}&\bar{0}&\bar{0}&\bar{0}&\bar{0}
\\
\bar{0}&\bar{1}&\bar{0}&\bar{0}&\bar{0}&\bar{1}&\bar{1}
\end{array}\right]
$
}&
{\tiny
$
\left[
\begin{array}{c}
3
\\
0
\\
\bar{0}
\\
\bar{1}
\end{array}
\right]
$
}
\\
\hline
24&
{\tiny
$
\begin{array}{c}
\CC[T_{1},T_{2},T_{3},T_{4},T_{5},T_{6},T_{7}]
 \\ 
 \hline
\left\langle\begin{array}{c}
T_{1}T_{2}+T_{3}T_{4}+T_{6}^2,
\\
T_{1}T_{2}+T_{5}^2+T_{7}^2
\end{array}\right\rangle
\end{array}
$
}
&
$\ZZ^2\times(\ZZ_2)^2$&
{\tiny
$
\left[\begin{array}{ccccccc}
-1&1&2&-2&0&0&0
\\
1&1&1&1&1&1&1
\\
\bar{0}&\bar{0}&\bar{0}&\bar{0}&\bar{1}&\bar{1}&\bar{0}
\\
\bar{0}&\bar{0}&\bar{1}&\bar{1}&\bar{0}&\bar{1}&\bar{0}
\end{array}\right]
$
}&
{\tiny
$
\left[
\begin{array}{c}
0
\\
3
\\
\bar 0
\\
\bar 1
\end{array}
\right]
$
}
\\
\hline
25&
{\tiny
$
\begin{array}{c}
\CC[T_{1},T_{2},T_{3},T_{4},T_{5},T_{6},T_{7}]
 \\ 
 \hline
\left\langle\begin{array}{c}
T_{1}T_{2}+T_{3}T_{4}+T_{6}^2,
\\
T_{1}T_{2}+T_{5}^2+T_{7}^2
\end{array}\right\rangle
\end{array}
$
}
&
$\ZZ^2\times(\ZZ_2)^2$&
{\tiny
$
\left[\begin{array}{ccccccc}
2&-2&-1&1&0&0&0
\\
1&1&1&1&1&1&1
\\
\bar{0}&\bar{0}&\bar{0}&\bar{0}&\bar{1}&\bar{1}&\bar{0}
\\
\bar{1}&\bar{1}&\bar{0}&\bar{0}&\bar{0}&\bar{1}&\bar{0}
\end{array}\right]
$
}&
{\tiny
$
\left[
\begin{array}{c}
0
\\
3
\\
\bar 0
\\
\bar 1
\end{array}
\right]
$
}
\\
\hline
26&
{\tiny
$
\begin{array}{c}
\CC[T_{1},T_{2},T_{3},T_{4},T_{5},T_{6},S_{1}]
 \\ 
 \hline
\left\langle\begin{array}{c}
T_{1}T_{2}+T_{3}^2+T_{4}^2+T_{5}^2,
\\
T_{3}^2+aT_{4}^2+T_{6}^2
\end{array}\right\rangle
\\
a\not =0, 1
\end{array}
$
}
&
$\ZZ^2\times(\ZZ_2)^3$&
{\tiny
$
\left[\begin{array}{ccccccc}
-1&1&0&0&0&0&1
\\
1&1&1&1&1&1&1
\\
\bar{1}&\bar{1}&\bar{1}&\bar{1}&\bar{1}&\bar{0}&\bar{0}
\\
\bar{1}&\bar{1}&\bar{0}&\bar{0}&\bar{1}&\bar{0}&\bar{0}
\\
\bar{1}&\bar{1}&\bar{1}&\bar{0}&\bar{0}&\bar{0}&\bar{0}
\end{array}\right]
$
}&
{\tiny
$
\left[
\begin{array}{c}
1
\\
3
\\
\bar 1
\\
\bar 1
\\
\bar{1}
\end{array}
\right]
$
}
\\
\hline
27&
{\tiny
$
\begin{array}{c}
\CC[T_{1},T_{2},T_{3},T_{4},T_{5},T_{6},S_{1}]
 \\ 
 \hline
\left\langle\begin{array}{c}
T_{1}^2+T_{2}T_{3}+T_{4}^2+T_{5}^2,
\\
T_{2}T_{3}+aT_{4}^2+T_{6}^2
\end{array}\right\rangle
\\
a\not =0, 1
\end{array}
$
}
&
$\ZZ^2\times(\ZZ_2)^3$&
{\tiny
$
\left[\begin{array}{ccccccc}
0&-1&1&0&0&0&1
\\
1&1&1&1&1&1&1
\\
\bar{1}&\bar{1}&\bar{1}&\bar{1}&\bar{1}&\bar{0}&\bar{0}
\\
\bar{0}&\bar{1}&\bar{1}&\bar{0}&\bar{1}&\bar{0}&\bar{0}
\\
\bar{1}&\bar{1}&\bar{1}&\bar{0}&\bar{0}&\bar{0}&\bar{0}
\end{array}\right]
$
}&
{\tiny
$
\left[
\begin{array}{c}
1
\\
3
\\
\bar 1
\\
\bar 1
\\
\bar{1}
\end{array}
\right]
$
}
\\
\hline
28&
{\tiny
$
\begin{array}{c}
\CC[T_{1},T_{2},T_{3},T_{4},T_{5},T_{6},S_{1}]
 \\ 
 \hline
\left\langle\begin{array}{c}
T_{1}T_{2}+T_{3}^2+T_{5}^2,
\\
T_{1}T_{2}+T_{4}^2+T_{6}^2
\end{array}\right\rangle
\end{array}
$
}
&
$\ZZ^2\times(\ZZ_2)^3$&
{\tiny
$
\left[\begin{array}{ccccccc}
-1&1&0&0&0&0&1
\\
1&1&1&1&1&1&1
\\
\bar{1}&\bar{1}&\bar{1}&\bar{1}&\bar{1}&\bar{0}&\bar{0}
\\
\bar{1}&\bar{1}&\bar{0}&\bar{0}&\bar{1}&\bar{0}&\bar{0}
\\
\bar{1}&\bar{1}&\bar{1}&\bar{0}&\bar{0}&\bar{0}&\bar{0}
\end{array}\right]
$
}&
{\tiny
$
\left[
\begin{array}{c}
1
\\
3
\\
\bar 1
\\
\bar 1
\\
\bar 1
\end{array}
\right]
$
}
\\
\hline
29&
{\tiny
$
\begin{array}{c}
\CC[T_{1},T_{2},T_{3},T_{4},T_{5},T_{6},S_{1}]
 \\ 
 \hline
\left\langle\begin{array}{c}
T_{1}^2+T_{2}T_{3}+T_{5}^2,
\\
T_{1}^2+T_{4}^2+T_{6}^2
\end{array}\right\rangle
\end{array}
$
}
&
$\ZZ^2\times(\ZZ_2)^3$&
{\tiny
$
\left[\begin{array}{ccccccc}
0&-1&1&0&0&0&1
\\
1&1&1&1&1&1&1
\\
\bar{1}&\bar{1}&\bar{1}&\bar{1}&\bar{1}&\bar{0}&\bar{0}
\\
\bar{0}&\bar{1}&\bar{1}&\bar{0}&\bar{1}&\bar{0}&\bar{0}
\\
\bar{1}&\bar{1}&\bar{1}&\bar{0}&\bar{0}&\bar{0}&\bar{0}
\end{array}\right]
$
}&
{\tiny
$
\left[
\begin{array}{c}
1
\\
3
\\
\bar 1
\\
\bar 1
\\
\bar 1
\end{array}
\right]
$
}
\\
\hline
30&
{\tiny
$
\begin{array}{c}
\CC[T_{1},T_{2},T_{3},T_{4},T_{5},T_{6},S_{1}]
 \\ 
 \hline
\left\langle\begin{array}{c}
T_{1}^2+T_{2}^2+T_{3}^2+T_{4}^2,
\\
T_{2}^2+aT_{3}^2+T_{5}T_{6}\\
a\not = 0,1
\end{array}\right\rangle
\end{array}
$
}
&
$\ZZ^2\times(\ZZ_2)^3$&
{\tiny
$
\left[\begin{array}{ccccccc}
0&0&0&0&-1&1&1
\\
1&1&1&1&1&1&1
\\
\bar{1}&\bar{1}&\bar{1}&\bar{0}&\bar{1}&\bar{1}&\bar{0}
\\
\bar{1}&\bar{1}&\bar{0}&\bar{0}&\bar{0}&\bar{0}&\bar{0}
\\
\bar{0}&\bar{1}&\bar{0}&\bar{0}&\bar{1}&\bar{1}&\bar{0}
\end{array}\right]
$
}&
{\tiny
$
\left[
\begin{array}{c}
1
\\
3
\\
\bar 1
\\
\bar 0
\\
\bar 1
\end{array}
\right]
$
}
\\
\hline
31&
{\tiny
$
\begin{array}{c}
\CC[T_{1},T_{2},T_{3},T_{4},T_{5},T_{6},T_{7}]
 \\ 
 \hline
\left\langle\begin{array}{c}
T_{1}^2+T_{2}T_{3}+T_{4}^2+T_{5}^2,
\\
T_{2}T_{3}+aT_{4}^2+T_{6}T_{7}\\
a\not = 0,1
\end{array}\right\rangle
\end{array}
$
}
&
$\ZZ^2\times(\ZZ_2)^2$&
{\tiny
$
\left[\begin{array}{ccccccc}
0&1&-1&0&0&-1&1
\\
1&2&0&1&1&1&1
\\
\bar{1}&\bar{0}&\bar{0}&\bar{1}&\bar{0}&\bar{0}&\bar{0}
\\
\bar{0}&\bar{0}&\bar{0}&\bar{1}&\bar{1}&\bar{0}&\bar{0}
\end{array}\right]
$
}&
{\tiny
$
\left[
\begin{array}{c}
0
\\
3
\\
\bar 0
\\
\bar 0
\end{array}
\right]
$
}
\\
\hline
32&
{\tiny
$
\begin{array}{c}
\CC[T_{1},T_{2},T_{3},T_{4},T_{5},T_{6},T_{7}]
 \\ 
 \hline
\left\langle\begin{array}{c}
T_{1}^2+T_{2}T_{3}+T_{4}^2+T_{5}^2,
\\
T_{2}T_{3}+aT_{4}^2+T_{6}T_{7}\\
a\not =0,1
\end{array}\right\rangle
\end{array}
$
}
&
$\ZZ^2\times(\ZZ_2)^2$&
{\tiny
$
\left[\begin{array}{ccccccc}
1&1&1&1&1&1&1
\\
0&-2&2&0&0&1&-1
\\
\bar{1}&\bar{0}&\bar{0}&\bar{1}&\bar{0}&\bar{0}&\bar{0}
\\
\bar{1}&\bar{1}&\bar{1}&\bar{0}&\bar{0}&\bar{0}&\bar{0}
\end{array}\right]
$
}&
{\tiny
$
\left[
\begin{array}{c}
3
\\
0
\\
\bar 0
\\
\bar 1
\end{array}
\right]
$
}
\\
\hline
33&
{\tiny
$
\begin{array}{c}
\CC[T_{1},T_{2},T_{3},T_{4},T_{5},T_{6},T_{7}]
 \\ 
 \hline
\left\langle\begin{array}{c}
T_{1}^2+T_{2}T_{3}+T_{4}^2+T_{5}^2,
\\
T_{2}T_{3}+aT_{4}^2+T_{6}T_{7}\\
a\not = 0,1
\end{array}\right\rangle
\end{array}
$
}
&
$\ZZ^2\times(\ZZ_2)^2$&
{\tiny
$
\left[\begin{array}{ccccccc}
0&-1&1&0&0&2&-2
\\
1&1&1&1&1&1&1
\\
\bar{1}&\bar{0}&\bar{0}&\bar{1}&\bar{0}&\bar{0}&\bar{0}
\\
\bar{1}&\bar{0}&\bar{0}&\bar{0}&\bar{0}&\bar{1}&\bar{1}
\end{array}\right]
$
}&
{\tiny
$
\left[
\begin{array}{c}
0
\\
3
\\
\bar 0
\\
\bar 1
\end{array}
\right]
$
}
\\
\hline
34&
{\tiny
$
\begin{array}{c}
\CC[T_{1},T_{2},T_{3},T_{4},T_{5},T_{6},T_{7}]
 \\ 
 \hline
\left\langle\begin{array}{c}
T_{1}^2+T_{2}T_{3}+T_{4}^2+T_{5}^2,
\\
T_{2}T_{3}+aT_{4}^2+T_{6}T_{7}\\
a\not = 0,1
\end{array}\right\rangle
\end{array}
$
}
&
$\ZZ^2\times(\ZZ_2)^2$&
{\tiny
$
\left[\begin{array}{ccccccc}
0&1&-1&0&0&-1&1
\\
1&1&1&1&1&0&2
\\
\bar{1}&\bar{0}&\bar{0}&\bar{1}&\bar{0}&\bar{0}&\bar{0}
\\
\bar{1}&\bar{1}&\bar{1}&\bar{0}&\bar{0}&\bar{0}&\bar{0}
\end{array}\right]
$
}&
{\tiny
$
\left[
\begin{array}{c}
0
\\
3
\\
\bar 0
\\
\bar 1
\end{array}
\right]
$
}
\\
\hline
35&
{\tiny
$
\begin{array}{c}
\CC[T_{1},T_{2},T_{3},T_{4},T_{5},T_{6},T_{7}]
 \\ 
 \hline
\left\langle\begin{array}{c}
T_{1}^2+T_{2}T_{3}+T_{5}T_{6},
\\
T_{1}^2+T_{4}^2+T_{7}^2
\end{array}\right\rangle
\end{array}
$
}
&
$\ZZ^2\times(\ZZ_2)^2$&
{\tiny
$
\left[\begin{array}{ccccccc}
0&1&-1&0&-1&1&0
\\
1&2&0&1&1&1&1
\\
\bar{1}&\bar{0}&\bar{0}&\bar{1}&\bar{0}&\bar{0}&\bar{0}
\\
\bar{0}&\bar{1}&\bar{1}&\bar{1}&\bar{0}&\bar{0}&\bar{0}
\end{array}\right]
$
}&
{\tiny
$
\left[
\begin{array}{c}
0
\\
3
\\
\bar 0
\\
\bar 1
\end{array}
\right]
$
}
\\
\hline
36&
{\tiny
$
\begin{array}{c}
\CC[T_{1},T_{2},T_{3},T_{4},T_{5},T_{6},T_{7}]
 \\ 
 \hline
\left\langle\begin{array}{c}
T_{1}^2+T_{2}T_{3}+T_{5}T_{6},
\\
T_{1}^2+T_{4}^2+T_{7}^2
\end{array}\right\rangle
\end{array}
$
}
&
$\ZZ^2\times(\ZZ_2)^2$&
{\tiny
$
\left[\begin{array}{ccccccc}
0&-1&1&0&0&0&0
\\
1&1&1&1&1&1&1
\\
\bar{1}&\bar{0}&\bar{0}&\bar{1}&\bar{0}&\bar{0}&\bar{0}
\\
\bar{0}&\bar{0}&\bar{0}&\bar{1}&\bar{1}&\bar{1}&\bar{0}
\end{array}\right]
$
}&
{\tiny
$
\left[
\begin{array}{c}
0
\\
3
\\
\bar 0
\\
\bar 1
\end{array}
\right]
$
}
\\
\hline
37&
{\tiny
$
\begin{array}{c}
\CC[T_{1},T_{2},T_{3},T_{4},T_{5},T_{6},T_{7}]
 \\ 
 \hline
\left\langle\begin{array}{c}
T_{1}^2+T_{2}T_{3}+T_{5}T_{6},
\\
T_{1}^2+T_{4}^2+T_{7}^2
\end{array}\right\rangle
\end{array}
$
}
&
$\ZZ^2\times(\ZZ_2)^2$&
{\tiny
$
\left[\begin{array}{ccccccc}
0&-1&1&0&2&-2&0
\\
1&1&1&1&1&1&1
\\
\bar{1}&\bar{0}&\bar{0}&\bar{1}&\bar{0}&\bar{0}&\bar{0}
\\
\bar{0}&\bar{0}&\bar{0}&\bar{1}&\bar{1}&\bar{1}&\bar{0}
\end{array}\right]
$
}&
{\tiny
$
\left[
\begin{array}{c}
0
\\
3
\\
\bar 0
\\
\bar 1
\end{array}
\right]
$
}
\\
\hline
38&
{\tiny
$
\begin{array}{c}
\CC[T_{1},T_{2},T_{3},T_{4},T_{5},T_{6},T_{7}]
 \\ 
 \hline
\left\langle\begin{array}{c}
T_{1}^2+T_{2}T_{3}+T_{5}T_{6},
\\
T_{1}^2+T_{4}^2+T_{7}^2
\end{array}\right\rangle
\end{array}
$
}
&
$\ZZ^2\times(\ZZ_2)^2$&
{\tiny
$
\left[\begin{array}{ccccccc}
0&1&-1&0&-1&1&0
\\
1&1&1&1&0&2&1
\\
\bar{1}&\bar{0}&\bar{0}&\bar{1}&\bar{0}&\bar{0}&\bar{0}
\\
\bar{0}&\bar{1}&\bar{1}&\bar{1}&\bar{0}&\bar{0}&\bar{0}
\end{array}\right]
$
}&
{\tiny
$
\left[
\begin{array}{c}
0
\\
3
\\
\bar 0
\\
\bar 1
\end{array}
\right]
$
}
\\
\hline
39&
{\tiny
$
\begin{array}{c}
\CC[T_{1},T_{2},T_{3},T_{4},T_{5},T_{6},T_{7}]
 \\ 
 \hline
\left\langle\begin{array}{c}
T_{1}^2+T_{2}T_{3}+T_{5}T_{6},
\\
T_{1}^2+T_{4}^2+T_{7}^2
\end{array}\right\rangle
\end{array}
$
}
&
$\ZZ^2\times(\ZZ_2)^2$&
{\tiny
$
\left[\begin{array}{ccccccc}
0&2&-2&0&-1&1&0
\\
1&1&1&1&1&1&1
\\
\bar{1}&\bar{0}&\bar{0}&\bar{1}&\bar{0}&\bar{0}&\bar{0}
\\
\bar{0}&\bar{1}&\bar{1}&\bar{1}&\bar{0}&\bar{0}&\bar{0}
\end{array}\right]
$
}&
{\tiny
$
\left[
\begin{array}{c}
0
\\
3
\\
\bar 0
\\
\bar 1
\end{array}
\right]
$
}
\\
\hline
40&
{\tiny
$
\begin{array}{c}
\CC[T_{1},T_{2},T_{3},T_{4},T_{5},T_{6},T_{7}]
 \\ 
 \hline
\left\langle\begin{array}{c}
T_{1}^2+T_{2}T_{3}+T_{5}T_{6},
\\
T_{1}^2+T_{4}^2+T_{7}^2
\end{array}\right\rangle
\end{array}
$
}
&
$\ZZ^2\times(\ZZ_2)^2$&
{\tiny
$
\left[\begin{array}{ccccccc}
0&-1&1&0&0&0&0
\\
2&2&2&2&1&3&2
\\
\bar{1}&\bar{0}&\bar{0}&\bar{1}&\bar{1}&\bar{1}&\bar{0}
\\
\bar{1}&\bar{0}&\bar{0}&\bar{0}&\bar{0}&\bar{0}&\bar{1}
\end{array}\right]
$
}&
{\tiny
$
\left[
\begin{array}{c}
0
\\
6
\\
\bar 0
\\
\bar 0
\end{array}
\right]
$
}
\\
\hline
\end{longtable}
}
\end{theorem}

We now describe the structure of this paper. 
In Section~\ref{sec:2}, we investigate special arrangement varieties using the approach to varieties with torus action from~\cite{HaHiWr2019}. In particular, we obtain an explicit description of their Cox rings, see Construction~\ref{constr:RAP}. Section~\ref{sec:3} is dedicated to the proofs to Section~\ref{sec:2}. In Section~\ref{sec:4} we turn to the smooth case. Here, we prove that there are no smooth honest special arrangement varieties up to Picard number two, see Theorem~\ref{thm:noSmoothOnes}.
This motivates our subsequent study of 
the singular case in Section~\ref{sec:5}.
Here we show that all arrangement varieties allow a toric ambient resolution of singularities as introduced in~\cite{HaHiWr2019}, which provides a method to desingularize arrangement varieties in purely combinatorial terms.
Section~\ref{sec:6} is dedicated to the anticanonical complex, which is a combinatorial tool to detect the singularity type of a variety, see~\cite{HiWr2018}.
We give an explicit description of this tool
in the arrangement case, extending results of~\cites{BeHaHuNi2016, HiWr2018}.
In Section~\ref{sec:7} we use this description to prove Theorem~\ref{thm:classification}.
Finally, in Section~\ref{sec:8} we go one step beyond the arrangement case and consider varieties with torus action
where the maximal orbit quotient decomposes as a product of projective spaces and the critical values form a hyperplane arrangement compatible to this decomposition.
We use this approach to detect hidden torus actions on special arrangement varieties, see Propostion~\ref{prop:notTrue}. Moreover, we provide full classification results for smooth varieties of this type in the Fano and more general in the projective case up to Picard number two, see Theorem~\ref{thm:someSmoothOnes}.

\tableofcontents

\section{Arrangement varieties and their Cox rings}\label{sec:2}
In this section we introduce the class of arrangement varieties. These are normal, irreducible varieties $X$ coming with an effective action of an algebraic torus $\TT\times X\rightarrow X$ and having a projective space as maximal orbit quotient such that the critical values form a hyperplane arrangement. 
In the first part of this section we recall the necessary notions. After the central definition of an arrangement variety, we turn to the description of their Cox rings and their realization as an explicit $\TT$-variety.
Accompanying the reader, we have the running Examples
\ref{ex:sec2Lauf1},~\ref{ex:sec2Lauf2},~\ref{ex:sec2Lauf3},~\ref{ex:sec2Lauf4},~\ref{ex:sec2Lauf5} and~\ref{ex:sec2Lauf6}. The proofs of the statements are presented in the subsequent section.

A {\em projective hyperplane arrangement} is a collection of hyperplanes $H_0,\ldots,H_r$ in a projective space $\PP_n$. We call a projective hyperplane arrangement {\em in general position}, if for every choice $0\leq i_1<\ldots<i_k\leq r$, the intersection $H_{i_1}\cap\ldots\cap H_{i_k}$ is of codimensionen $k$. Otherwise we call it {\em in special position}.

\begin{example}\label{ex:sec2Lauf1}
Consider the following collection of lines in $\PP_2$:
$$H_0:=V(T_0),\qquad H_1:=V(T_1),\qquad H_2:=V(T_2)$$
$$H_3:=V(T_0+T_1),\qquad H_4:=V(T_0+T_2).$$
Then $H_0,\ldots,H_r$ is a projective hyperplane arrangement in special position.
\end{example}

Let $X$ be an algebraic variety with an effective action of an algebraic torus $\TT\times X\rightarrow X$. 
A {\em maximal orbit quotient} for the $\TT$-action on $X$ is a rational quotient, i.e.\ a dominant rational map $\pi\colon X\dashrightarrow Y$ with $\CC(Y)=\CC(X)^\TT$,
together with a surjective representative $\psi\colon W\rightarrow V$ and a collection of prime divisors $C_0,\ldots,C_r$ on $Y$ having the following properties:
\begin{enumerate}
    \item The set $W$ is contained in the (open) subset $X_0\subseteq X$ of points with at most finite $\TT$-isotropy and the complements of the sets $W\subseteq X_0$ and $V\subseteq Y$ are of codimension at least two.
    \item Every preimage $\psi^{-1}(C_i)$ is a union of prime divisors $D_{i1}\cup\cdots\cup D_{in_i}\subseteq W$.
    \item All $\TT$-invariant prime divisors of $X_0$ with non-trivial generic isotropy group occur among the $D_{ij}$.
    \item Every choice of divisors $D_{0j_0},\ldots,D_{rj_r}$ defines a geometric quotient ${\psi\colon (W\setminus \cup_{j\not = j_i}D_{ij})\rightarrow V}$ for the $\TT$-action on $X$.
\end{enumerate}

\begin{definition}\label{def:arrangementVariety}
A {\em general (special) arrangement variety} is a variety $X$ with an effective torus action $\TT\times X\rightarrow X$ having
$\pi\colon X\dashrightarrow \PP_c$ as a maximal orbit quotient and the critical values form a projective hyperplane arrangement in general (special) position.
\end{definition}

Note that the case of general arrangement varieties were studied starting in the complexity one case~\cites{HaHeSu2011, HaHe2013} and are more generally treated in~\cite{HaHiWr2019} in the case of higher complexity. We will extend this picture in the next sections by studying special arrangement varieties

\begin{remark}
Every arrangement variety has a finitely generated Cox ring, due to~\cite{HaSu2010}*{Thm.~1.2}.
\end{remark}

Having a finitely generated Cox ring, we turn to the description of arrangement varieties as {\em explicit $\TT$-varieties}~\cite{HaHiWr2019}*{Sec. 3}. We work in a similar manner as done in the case of general arrangement varieties~\cite{HaHiWr2019}*{Sec. 6} and start with the description of their Cox rings.

\begin{construction}\label{constr:RAP0}
Fix integers $r \geq c > 0$, 
$n_0,\ldots,n_r>0$ and $m\geq 0$
and set $n:=n_0+\ldots + n_r$.
The input data is a tuple $(A,P_0)$ as follows:
\begin{enumerate}
    \item[\textbullet] $A=(a_0,\ldots,a_r)$ is a $(c+1)\times (r+1)$ matrix of full rank with pairwise linearly independent columns $a_i$.
    \item[\textbullet] 
    $P_0$ is a $r\times (n+m)$ matrix build up from tuples $l_i=(l_{i1},\ldots,l_{in_i})$ of positive integers
    $$P_0=\left[
    \begin{array}{ccccccc}
    -l_0&l_1&&&0&\ldots&0\\
    \vdots&&\ddots&&\vdots&&\vdots\\
    -l_0&&&l_r&0&\ldots&0
    \end{array}
    \right]$$
\end{enumerate}
Write  $\CC[T_{ij},S_k]$ for the polynomial ring 
in the variables $T_{ij}$, where $i = 0, \ldots, r$, 
$j = 1, \ldots, n_i$, 
and $S_k$, where $k = 1, \ldots, m$.
For every $l_i$ we define a monomial 
$$
T_i^{l_i}
\ := \ 
T_{i1}^{l_{i1}} \cdots T_{in_i}^{l_{in_i}}
\ \in \ 
\CC[T_{ij},S_k].
$$
and to any $v\in\CC^{r+1}$ we assign the polynomial
$$g_v:=v_0T_0^{l_0}+\ldots+v_r T_r^{l_r}\in\CC[T_{ij},S_k].$$
Then any tuple $(A,P_0)$
defines a \emph{$\CC$-algebra}
$$ 
\CC[T_{ij},S_k] / \bangle{g_v; \ v\in\mathrm{ker}(A)}.
$$
Now let $e_{ij} \in \ZZ^{n}$ 
and $e_k \in \ZZ^{m}$ denote the 
canonical basis vectors
and let 
$$
Q_0 \colon \ZZ^{n+m} 
\ \to \ 
K_0 := \ZZ^{n+m} / \im(P_0^*)
$$ 
be the projection 
onto the factor group
by the row lattice of $P_0$.
This defines a $K_0$-graded $\CC$-algebra
$$R(A,P_0) := \CC[T_{ij},S_k] / \bangle{g_v; \ v\in\mathrm{ker}(A)}$$
$$\deg(T_{ij}):=Q_0(e_{ij}),\qquad
\deg(S_k):=Q_0(e_k).$$
\end{construction}

\begin{example}\label{ex:sec2Lauf2}
Consider the projective hyperplane arrangement 
$H_0,\ldots,H_4$ from Example~\ref{ex:sec2Lauf1}.
We store the coefficients of the defining linear forms 
as the columns of a matrix $A$ and set
$$A:=
\left[
\begin{array}{ccccc}
1&0&0&1&1\\
0&1&0&1&0\\
0&0&1&0&1
\end{array}
\right],\quad
P_0:=
\left[
\begin{array}{ccccccc}
-1&-1&2&0&0&0&0\\
-1&-1&0&2&0&0&0\\
-1&-1&0&0&2&0&0\\
-1&-1&0&0&0&2&0
\end{array}
\right].$$
Using this as input data in Construction~\ref{constr:RAP0}, we obtain the following $\ZZ^3\times (\ZZ_2)^3$ graded $\CC$-algebra, where we store 
the degrees of the genrators $T_{ij}$ and $S_k$ as the columns of a matrix $Q_0$:
$$
R(A,P_0)=
     \CC[T_{01},T_{02},T_{11},T_{21},T_{31},T_{41},S_1]/
     \bangle{
T_{01}T_{02}+T_{11}^2+T_{31}^2,
T_{01}T_{02}+T_{21}^2+T_{41}^2
}
$$
$$
Q_0=\left[\begin{array}{ccccccc}
-1&1&0&0&0&0&0
\\
0&0&0&0&0&0&1
\\
1&1&1&1&1&1&0
\\
\bar{0}&\bar{0}&\bar{1}&\bar{1}&\bar{1}&\bar{0}&\bar{0}
\\
\bar{0}&\bar{0}&\bar{1}&\bar{1}&\bar{0}&\bar{0}&\bar{0}
\\
\bar{0}&\bar{0}&\bar{1}&\bar{0}&\bar{1}&\bar{0}&\bar{0}
\end{array}\right]
$$
A direct computation shows that the ring $R(A,P_0)$ is an integral, normal, complete intersection ring of dimension $5$ and has $R(A,P_0)^* = \CC^*$.
\end{example}

In order to state the basic properties of the $K_0$-graded algebras $R(A,P_0)$, we recall the necessary definitions:
Let $K$ be a finitely generated abelian group and consider a $K$-graded $\CC$-algebra $A=\oplus_{w\in K}A_w$.
We say that $A$ is {\em $K$-integral},  if it has no homogeneous zero divisors. 
A {\em $K$-prime} element of $A$ is a non-zero homogeneous non-unit $f\in A$ that, whenever it divides a product of homogeneous elements, it divides one of the factors. We say that $A$ is $K$-factorial if $A$ is $K$-integral and every non-zero homogeneous non-unit is a product of $K$-primes. 
Moreover, we call the $K$-grading on $A$ {\em effective}, if the weights $w\in K$ with $A_w\not = \{0\}$ generate $K$ as a group and {\em pointed} if $A_w \not = \{0\}\not = A_{-w}$ can hold only for torsion elements $w\in K$ and in addition $A_0 = \CC$ holds. Finally we call $c:=\dim(A)-\rk(K)$ the {\em complexity} of the grading.

\begin{theorem}\label{thm:arrIntNormal}
Let $R(A,P_0)$ be a $\CC$-algebra 
arising from Construction~\ref{constr:RAP0}.
Then $R(A,P_0)$ is an integral, normal, complete intersection 
ring satisfying 
$$ 
\dim(R(A,P_0)) \ = \ n+m-r+c,
\qquad 
R(A, P_0)^* \ = \ \CC^*.
$$
The $K_0$-grading is effective, pointed, factorial and of complexity~$c$.
\end{theorem}

\begin{construction}\label{constr:RAP}
Let $R(A,P_0)$ be as in Construction~\ref{constr:RAP0}.
We build up a new $(r+s)\times(n+m)$ matrix 
$$P:=\left[\begin{array}{c}
P_0\\
d
\end{array}\right],$$
where we require the columns of $P$ to be pairwise different, primitive, generating $\QQ^{r+s}$ as a vectorspace. 
Let $Q\colon \ZZ^{n+m}\rightarrow \ZZ^{n+m}/\im(P^*)=:K$ be the canonical projection. Then we obtain a new graded ring $R(A,P)$ by defining a $K$-grading on the ring $R(A,P_0)$ by setting
$$R(A,P) := \CC[T_{ij},S_k] / \bangle{g_v; \ v\in\mathrm{ker}(A)}$$
$$\deg(T_{ij}):=Q(e_{ij}),\qquad \deg(S_k):=Q(e_k).$$
\end{construction}

\begin{example}\label{ex:sec2Lauf3}
We continue Example~\ref{ex:sec2Lauf2} and choose
$$
d:=
\left[
\begin{array}{ccccccc}
-2&-3&1&1&1&1&1
\end{array}
\right].
$$
This defines a $\ZZ^{2}\times(\ZZ_2)^3$
grading on the resulting algebra
$$
R(A,P)=\CC[T_{01},T_{02},T_{11},T_{21},T_{31},T_{41},S_1] / \bangle{
T_{01}T_{02}+T_{11}^2+T_{31}^2,
T_{01}T_{02}+T_{21}^2+T_{41}^2
},
$$
where we store the degrees of the variables as the columns of the matrix $Q$:
$$
Q=\left[\begin{array}{ccccccc}
-1&1&0&0&0&0&1
\\
1&1&1&1&1&1&1
\\
\bar{1}&\bar{1}&\bar{1}&\bar{1}&\bar{1}&\bar{0}&\bar{0}
\\
\bar{1}&\bar{1}&\bar{0}&\bar{0}&\bar{1}&\bar{0}&\bar{0}
\\
\bar{1}&\bar{1}&\bar{1}&\bar{0}&\bar{0}&\bar{0}&\bar{0}
\end{array}\right].
$$

\end{example}

\begin{corollary}
The $K$-grading on $R(A,P)$ is effective, pointed, factorial and of complexity $c$.
Moreover, if the columns of $P$ generate $\QQ^{r+s}$ as a cone, then the $K$-grading is pointed.
\end{corollary}

The following proposition and the subsequent statements are adapted versions or direct consequences of~\cite{HaHiWr2019}*{Secs. 3 and 6}.

\begin{proposition}\label{prop:CoxRing}
The Cox ring of an arrangement variety is isomorphic to a ring $R(A,P)$ as in Construction~\ref{constr:RAP}.
\end{proposition}

Now we turn to the realization of arrangement varieties
as explicit $\TT$-varieties using the rings $R(A,P)$.
Recall that $R(A,P)$ is an irreducible, normal complete intersection ring of dimension $n+m-(r-c)$.
In particular, it defines an affine variety $V(g_1,\ldots,g_{r-c})$, where $g_1,\ldots,g_{r-c}$ are generators for the 
ideal $\bangle{g_v,v\in\Ker(A)}$ as in Construction~\ref{constr:RAP}.

\begin{construction}\label{constr:XAPSigma}
Let $R(A,P)$ be a $K$-graded ring as in Construction~\ref{constr:RAP} and assume the variables $T_{ij},S_k$ to be $K$-prime.
Choose any fan $\Sigma$ in $\ZZ^{r+s}$ having precisely the columns of $P$ as its primitive ray generators
and denote by $Z$ the corresponding toric variety.
Then we obtain the following diagram
\begin{center}
\begin{tikzcd}
V(g_1,\ldots,g_{r-c})\hspace{-20pt}
&
=:\hspace{-20pt}
&
\bar{X}\arrow[r,hook]
&
\bar{Z}
&
\hspace{-20pt}
:=
&
\hspace{-20pt}
\CC^{n+m} 
\\[-20pt]
&&
\rotatebox{90}{$\subseteq$}
&
\rotatebox{90}{$\subseteq$}
&&
\\[-20pt]
\bar{X} \cap \hat{Z}\hspace{-20pt}
&
=:\hspace{-20pt}
&
\hat{X}\ar[d,"\quot H"]\arrow[r,hook]
&
\hat{Z}\ar[d,"\quot H"]
&
&
\\
&
&
X\arrow[r,hook]
&
Z
&
&
\end{tikzcd}
\end{center}
where $H:=\Spec\ \CC(K)$
is the characteristic quasitorus of $Z$, acting on the 
characteristic space $\hat{Z}\rightarrow Z$ and $X(A,P,\Sigma):=X$ is the image of $\hat{X}$ under the latter morphism.
The torus $T$ acting on $Z$ splits
as a product $T^r\times T^s$ and the $\TT:=T^s$-factor
leaves $X\subseteq Z$ invariant.
\end{construction}

\begin{remark}
The varieties $X:=X(A,P,\Sigma)$ as in Construction~\ref{constr:XAPSigma} are irreducible and normal with dimension, invertible functions, divisor class group and Cox ring given by
$$
\dim(X)=s+c,\quad \Gamma(X,\OOO^*)=\CC^*,\quad
\Cl(X) = K,\quad \RRR(X)=R(A,P).
$$ 
Moreover the $\TT$-action on $X$ is
effective and of complexity $c$.
\end{remark}

\begin{example}\label{ex:sec2Lauf4}
We continue Example~\ref{ex:sec2Lauf3}. 
Note that the variables $T_{ij}$ and $S_1$ of $R(A,P)$ are $K$-prime. Denoting the columns of $P$ with $v_{ij}$ and $v_k$ with respect to the variables $T_{ij}$ and $S_k$ we choose the fan $\Sigma$ with maximal cones 
$$
\cone(v_{02},v_{11},v_{21},v_{31},v_{41}),\
\cone(v_{01},v_{21},v_{41},v_1),
$$
$$
\cone(v_{01},v_{11},v_{31},v_1),\
\cone(v_{01},v_{02},v_{21},v_{41}),\
\cone(v_{01},v_{02},v_{11},v_{31}),
$$
$$
\cone(v_{31},v_{41},v_1),\
\cone(v_{21},v_{31},v_1),\
\cone(v_{11},v_{41},v_1),\
\cone(v_{11},v_{21},v_1).
$$
The resulting variety $X(A,P,\Sigma)$ 
has dimension three, only constant invertible global functions, divisor class group $\Cl(X)\cong \ZZ^2\times (\ZZ_2)^3$ and Cox ring $\RRR(X) = R(A,P)$.
Moreover the $\TT$-action is of complexity two. 
\end{example}

\begin{construction}\label{constr:MOQ}
Let $X:=X(A,P,\Sigma)\subseteq Z$ be as in Construction~\ref{constr:XAPSigma}.
Then $X$ fits into the following diagram:
\begin{center}
\begin{tikzcd}
X\arrow[r, hook] \arrow[d,dashed]
& 
Z\arrow[d, dashed]
&
\supseteq
&
Z_0\arrow[lld]
\\
\PP_c\arrow[r, hook]
&
\PP_r
\end{tikzcd}
\end{center}
where $Z_0$ is the (open) union of the torus and all orbits of codimension one in $Z$,
the morphism $Z_0\rightarrow \PP_r$ is a toric morphism induced by the projection of tori $T^{r+s}\rightarrow T^{r}$, the downward rational maps are defined via this morphism and $\PP_c$ is linearly embedded into $\PP_r$ via $[x]\mapsto[A^t x]$.
\end{construction}

\begin{remark}
The rational map $X\dashrightarrow \PP_c$
is a maximal orbit quotient for the $\TT$-action of $X$, where the critical values form the hyperplane arrangement
$$
H_0,\ldots, H_r\subseteq \PP_c,\qquad
H_i:= \{[x]\in\PP_c;\bangle{a_i,x}=0\}.
$$
In particular any variety $X(A,P,\Sigma)$ as in Construction~\ref{constr:XAPSigma} is an arrangement variety.
\end{remark}

\begin{example}\label{ex:sec2Lauf5}
We continue Example~\ref{ex:sec2Lauf4}. The variety $X=X(A,P,\Sigma)$ is an arrangement variety having $X\dashrightarrow\PP_2$ as a maximal orbit quotient.
In this case $\PP_2$ is realized inside $\PP_4$ as 
$$\PP_2 = V(T_0+T_1+T_3,T_0+T_2+T_4)\subseteq\PP_4$$
and the critical values form the line arrangement in special position from Example~\ref{ex:sec2Lauf1}:
$$H_0 = V(T_0),\qquad H_1 = V(T_1),\qquad H_2 = V(T_2)$$
$$H_3 = V(T_0+T_1),\qquad H_4 =V(T_0+T_2).$$
In particular $X$ is an arrangement variety and as we will see later it is one of the three-dimensional Fano canonical complexity two varieties in Theorem~\ref{thm:classification}.
\end{example}

\begin{definition}
We call a variety $X(A,P,\Sigma)\subseteq Z$ as in Construction~\ref{constr:XAPSigma} an explicit arrangement variety.
\end{definition}

\begin{remark}
Every arrangement variety is equivariantly isomorphic to an explicit arrangement variety, see~\cite{HaHiWr2019}*{Thm. 3.10}.
\end{remark}

Let us recall the basic notions on tropical varieties.
For a closed subvariety $X\subseteq Z$ intersecting the torus non trivially consider the vanishing ideal $I(X \cap T)$ in the Laurent polynomial ring $\OOO(T)$. For every $f \in I(X\cap T)$ let $|\Sigma(f)|$ denote the support of the codimension one skeleton of the normal quasifan of its Newton polytope, 
where a quasifan is a fan, 
where we allow the cones to be non-pointed.
Then the \emph{tropical variety $\trop(X)$ of} $X$ is defined as follows, 
see~\cite{MaSt2015}*{Def. 3.2.1}:
$$
\trop(X) := \bigcap_{f \in I(X \cap T)} |\Sigma(f)| \subseteq \QQ^{\dim(Z)}.
$$

\begin{definition}\label{def:Pcone}
Let $X(A,P,\Sigma)\subseteq Z$ be an explicit arrangement variety. 
We denote the columns of $P$ with $v_{ij}$ and $v_k$ according to the variables $T_{ij}$ and $S_k$.
A {\em $P$-cone} is a cone $\sigma\subseteq\QQ^{r+s}$ 
such that its set of primitive ray generators is a subset of the columns of $P$, i.e.\
$$
\sigma=\cone(v_{ij_i},v_k;\ i\in I\subseteq\{0,\ldots,r\},\ j_i\in J_i\subseteq\{1,\ldots,n_i\}, k\in K\subseteq\{1,\ldots,m\}).
$$
We call a $P$-cone $\sigma\subseteq\QQ^{r+s}$
\begin{enumerate}
    \item a {\em leaf cone}, if $\sigma\subseteq|\trop(X)|$ holds.
    \item a {\em big cone}, if
    $\sigma^\circ\cap(\{0\}\times\QQ^s)\not = \emptyset$ holds.
    \item a {\em special cone}, if it is neither big nor leaf but
     $\sigma^\circ\cap|\trop(X)|\not = \emptyset$ holds.
\end{enumerate}
\end{definition}

Applying~\cite{Te2007}*{Lem. 2.2}, we obtain the following remark.

\begin{remark}
Let $X(A,P, \Sigma) \subseteq Z$ be an explicit arrangement variety. Then the cones in $\Sigma$ are of leaf, special or big~type.
\end{remark}

\begin{example}\label{ex:sec2Lauf6}
We continue Example~\ref{ex:sec2Lauf5} and 
investigate the fan $\Sigma$.
To describe the tropical variety of $X$ denote by 
$e_1,\ldots,e_4$ the canonical basis vectors of $\QQ^4$ and set
$$e_0:=-e_1-\ldots-e_4,\qquad e_5:=e_0+e_2+e_4,\qquad e_6:=e_0+e_1+e_3$$
and define a fan $\Delta$ with maximal cones
$\cone(e_i,e_j)$, where $(i,j)$ is one of the following tuples:
$$
(0,5),\ (0,6),\ (1,2),\ (1,4),\ (1,6),\ (2,3),\ (2,5),\ (3,4),\ (3,6),\ (4,5).
$$
Then $\trop(X)=|\Delta\times\QQ|$ holds.
Checking the items in Definition~\ref{def:Pcone} for the cones in $\Sigma$, we obtain one big cone
$$
\cone(v_{02},v_{11},v_{21},v_{31},v_{41}),
$$
four special cones
$$
\cone(v_{01},v_{21},v_{41},v_1),\
\cone(v_{01},v_{11},v_{31},v_1)
$$
$$
\cone(v_{01},v_{02},v_{21},v_{41}),\
\cone(v_{01},v_{02},v_{11},v_{31})
$$
and four leaf cones
$$
\cone(v_{31},v_{41},v_1),\
\cone(v_{21},v_{31},v_1),
$$
$$
\cone(v_{11},v_{41},v_1),\
\cone(v_{11},v_{21},v_1).
$$

\end{example}

\section{Proofs to Section~\ref{sec:2}}\label{sec:3}
This section is dedicated to the proof of the statements in Section~\ref{sec:2}. 
In a first step we investigate product structures on the rings $R(A,P_0)$. Then we turn to the proof of Theorem~\ref{thm:arrIntNormal}.

\begin{definition}\label{def:indecomposable}
Let $R(A,P_0)$ be a ring as in 
Construction~\ref{constr:RAP0}. 
\begin{enumerate}
    \item 
    We call the matrix $A$ \emph{indecomposable}
    if for any subset $I \subseteq \left\{1, \ldots, r+1\right\}$ we have $\left\{0\right\} \neq \mathrm{Lin}(a_i; \ i \in I) \cap \mathrm{Lin}(a_j; \ j \notin I).$
    \item  
    We call a ring 
    $R(A,P_0)$ \emph{indecomposable}
    if $A$ is indecomposable and $l_{ij}n_i > 1$ holds for all $i$.
\end{enumerate}
\end{definition}

\begin{proposition}\label{prop:decomp}
Every $\CC$-algebra $R(A,P_0)$ from Construction~\ref{constr:RAP0} is isomorphic as a $\CC$-algebra (forgetting the $K_0$-grading)
to a product  
\begin{equation}\label{equ:decomposition}
\bigotimes\limits_{i=1}\limits^t R(A^{(i)},P_0^{(i)})
\otimes \CC[S_1,\ldots,S_{m'}],
\end{equation}
where the algebras 
$R(A^{(i)},P_0^{(i)})$
are indecomposable with $m^{(i)} = 0$ for all $i=1,\ldots,t$ and $m'\geq m$ holds.
\end{proposition}

\begin{remark}
Note that these algebras are in general not isomorphic 
as graded algebras concerning their natural gradings: The $K_0$-grading defined on the product via the isomorphism is in general coarsening the grading defined via the product $K_0^{(1)}\times\ldots\times K_0^{(t)}\times\ZZ^{m'}$. We will investigate this fact in Section~\ref{sec:8}.
\end{remark}

\begin{remark}\label{rem:admissible}
The following list of {\em admissible operations} does not effect the isomorphy type of a ring $R(A,P_0)$:
\begin{enumerate}
    \item any elementary row  operation on $A$.
    \item swap columns in $A$ and accordingly columns in $P_0$.
    \item swap any column in $P_0$ inside a block $l_i$.
\end{enumerate}
In particular without loss of generality we may always assume the matrix $A$ to be in reduced row echelon form $A=(E_{c+1},a_{c+1},\ldots,a_r)$
and the polynomials $g_i$ generating $\mathrm{Ker}(A)$ to be of the following form: 
\begin{equation}\label{equ:relations}
g_i := \lambda_{0,i}T_0^{l_0} + \ldots + \lambda_{c,i}T_{c}^{l_{c}} + \lambda_{(c+i),i}T_{(c+i)}^{l_{(c+i)}},\qquad 1\leq i\leq r-c.    
\end{equation}
Note that we have $g_i = g_{v_i}$ for $v_i = (a_{c+1+i},-e_{i})$, where $e_i$ denotes the $i$-th canonical basis vector of $\CC^{r-c}$.
\end{remark}

We turn to the proof of Proposition~\ref{prop:decomp}.
The following lemma is straightforward but for the convenience of the reader we will prove it here:
\begin{lemma}\label{lem:decomp}
Let $A = (a_0, \ldots, a_r)$ be a matrix as in Construction~\ref{constr:RAP0}. Then there exists a unique decomposition of $\CC^{c+1}$
into vectorsubspaces $V_1 \oplus \ldots \oplus V_t$
such that the following holds:
\begin{enumerate}
    \item 
    For each $0 \leq i \leq r$ there exists $j(i) \in \left\{1, \ldots, t\right\}$ with $a_i \in V_{j(i)}$.
    \item
    If $V_1' \oplus \ldots \oplus V_s'$ is any other decomposition fulfilling {\rm{(i)}}, then
    for every $1 \leq i \leq t$ there exists $1 \leq j \leq s$ with 
    $V_i \subseteq V_j'$.
\end{enumerate}
\end{lemma}
\begin{proof}
Let $V_1 \oplus \cdots \oplus V_t$ be any decomposition of $\CC^{c+1}$ fulfilling \rm{(i)}. 
We construct a decomposition fulfilling (ii) by successively refining this given decomposition.
For this let $V_1' \oplus \ldots \oplus V_s'$ be any other decomposition fulfilling \rm{(i)} and assume 
our given decomposition does not fulfill (ii).
Then there exists $1 \leq i \leq s$ such that 
$V_i \not \subseteq V_j'$ for all $1 \leq j \leq s$. Set $A_i := \left\{k;  \ a_k \in V_i\right\}$.
Then for every $k \in A_i$ there exists $j(k)$ with
$a_k \in V_{j(k)}'.$
In particular, we obtain a decomposition
$$
V_i = 
V_i \cap (\bigoplus_{k \in A_i} V_{j(k)}') 
=
\bigoplus_{k \in A_i} V_i \cap V_{j(k)}',
$$
and $\mathrm{dim}(V_i) > \dim(V_i \cap V_{j(k)})$ holds for any $k \in A_i$. 
Iterating this step we end up with a decomposition fulfilling (ii). 
\end{proof}

Proposition~\ref{prop:decomp}
is a direct consequence of the following more technical
Lemma.

\begin{lemma}
Let $R(A,P_0)$ be a ring as in Construction $\ref{constr:RAP0}$.
Then the following statements hold:
\begin{enumerate}
    \item 
    Let $V_1 \oplus V_2 = \CC^{c+1}$ be a decomposition fulfilling assertion (i) of Lemma~\ref{lem:decomp} and assume $\dim(V_1) = 1$.
    Then $R(A,P_0) \cong R(A',P_0')$ holds, for a tuple $(A',P_0')$ with $\mathrm{rk}(A') < \mathrm{rk}(A)$ and $m'>m$.
    \item 
    Let $V_1 \oplus V_2 = \CC^{c+1}$ be a decomposition fulfilling assertion (i) of 
    Lemma~\ref{lem:decomp} and assume $\dim(V_i) > 1$ for $i = 1,2$.
    Then we have  
    $$R(A,P_0) \cong R(A^{(1)},P_0^{(1)}) \otimes R(A^{(2)},P_0^{(2)}),$$
    for suitably chosen data $(A^{(i)},P_0^{(i)})$ with $\mathrm{rk}(A^{(1)})+\mathrm{rk}(A^{(2)}) = \mathrm{rk}(A)$ holds.
\end{enumerate}
\end{lemma}

\begin{proof}
Let $V_1 \oplus V_2 = \CC^{c+1}$ be a
decomposition fulfilling Assertion (i) of Lemma~\ref{lem:decomp}.
Then, by applying Remark~\ref{rem:admissible}~(ii)
we may assume $V_1 = \mathrm{Lin}(a_0, \ldots, a_t)$
and $V_2 = \mathrm{Lin}(a_{t+1}, \ldots, a_r)$.
Furthermore, as elementary row operation do note effect the isomorphy type of $R(A,P_0)$, we may assume 
$V_1 = \mathrm{Lin}(e_1, \ldots, e_s)$
and
$V_2 = \mathrm{Lin}(e_{s+1}, \ldots, e_{c + 1})$.

We prove (i).
For this let $\mathrm{dim}(V_1) = 1$, i.e.\ we have $t=0$ and $s =1$. Then 
$a_0 = \lambda e_1$ holds and all
entries $a_{1j}$ with $j \geq 1$ equal zero.
We conclude 
$R(A,P_0)  \cong R(A',P_0')$,
with $m'=m + n_0$, $A'$ is the matrix obtained by deleting the first row and the first column of $A$ 
and $P_0'$ is build up from the tuples $l_1, \ldots, l_r$.

We turn to (ii). Assume we have $\dim(V_i) > 1$ for $i=1,2$. 
Then $A$ is a block matrix of the form  
$$
A = \left[
\begin{array}{cc}
     A^{(1)}&0  \\
     0&A^{(2)} 
\end{array}
\right]
$$ 
and we conclude $R(A,P_0) \cong R(A^{(1)},P_0^{(1)}) \otimes R(A^{(2)},P_0^{(2)})$, 
where $P_0^{(1)}$ is build up from $l_0, \ldots, l_t$,
$P_0^{(2)}$ is build up from $l_{t+1}, \ldots, l_r$
and $m_1, m_2$ are positive integers with
$m_1 +m_2 = m$.
\end{proof}

We turn to the proof of Theorem~\ref{thm:arrIntNormal}. 
Let $(A,P_0)$ be as in Construction~\ref{constr:RAP0}. Then the defining relations $g_v$ of the ring $R(A,P_0)$ can be obtained in the following way:
For any $v\in \Ker(A)$ write
$$f_v := v_0T_0+\ldots v_rT_r\in\CC[T_0,\ldots,T_r]$$ for the corresponding linear form. Then $g_v = f_v(T_0^{l_0},\ldots,T_r^{l_r})$ holds. 
We will use this observation to prove in Lemma~\ref{lem:connected} 
connectedness of the affine variety $X := V(g_v;\ v\in\Ker(A))$.
Moreover, in Proposition~\ref{prop:Xpure}
we deduce the dimension of $X$ from that of $Y := V(f_v;\ v\in\Ker(A))$.

\begin{remark}\label{rem:quasihom}
Consider the polynomial ring $\CC[T_0, \ldots, T_r]$ endowed with an effective pointed $\ZZ$-grading
$\deg(T_i) = w_i\in \ZZ_{>0}$ and let $f \in \CC[T_0, \ldots, T_r]$ be any homogeneous polynomial. 
Then the polynomial 
$$g := f(T_0^{l_0}, \ldots, T_r^{l_r}) \in \CC[T_{ij}, S_k]$$
is homogeneous with respect to the grading defined by:
\begin{equation}\label{equ:degTij}
\mathrm{deg}(T_{ij}) 
=
\frac
{
n_0 \ldots n_r l_{01} \ldots l_{rn_r} 
}
{
n_il_{ij}
}
\cdot w_i\in\mathbb{Z}_{>0}.
\end{equation}
\end{remark}

\begin{lemma}\label{lem:connected}
In the situation of Remark~\ref{rem:quasihom} let $f_1, \ldots, f_s \in \CC[T_0, \ldots, T_r]$ be homogeneous polynomials
and set
$g_i := f_i(T_0^{l_0},\ldots,T_r^{l_r})\in\CC[T_{ij},S_k]$.
Then the affine variety
$X:=\mathrm{V}(g_1,\ldots,g_s)$
is connected.
\end{lemma}
\begin{proof}
Consider the acting torus $(\CC^*)^{n+m}$ 
of $\CC^{n+m}$ and the multiplicative one-parameter subgroup
$$\lambda \colon \CC^* \rightarrow (\CC^*)^{n+m},
\qquad
t \mapsto (t^{\zeta_{01}}, \ldots, t^{\zeta_{rn_r}}, t, \ldots, t),
$$
where $\zeta_{ij} := \deg(T_{ij})$ is as in (\ref{equ:degTij}).
Then by construction the image $\lambda(\CC^*)$ acts on $X$ and has $0$ as an attractive fixed point. This gives the assertion.
\end{proof}

\begin{proposition}\label{prop:Xpure}
Let $Y = V(f_1,\ldots,f_s) \subseteq \CC^{r+1}$ 
be irreducible of dimension $r+1-s$ and set
$$
X:=V(g_1,\ldots,g_s)\ \text{ with }\ g_i := f_i(T_0^{l_0},\ldots,T_r^{l_r})\in\CC[T_{ij},S_k].$$
Then $X$ is pure of dimension $n+m-s$.
\end{proposition}
\begin{proof}
Let $X=X_1\cup\ldots\cup X_t$ be the decomposition of $X$ into irreducible components. Note that we have $\dim(X_j)\geq n+m-s$ for $1 \leq j \leq t$ as $X=V(g_1,\ldots,g_s)$ holds.
Consider the surjective morphism 
$$
\varphi\colon \CC^{n+m}\rightarrow \CC^{r+1}
\qquad
(x_{01}, \ldots, x_{rn_r}, x_1, \ldots x_m) \mapsto (x_0^{l_0}, \ldots, x_r^{l_r}).
$$
Then by construction of $X$ the restriction $\varphi|_X\colon X\rightarrow Y$ is again surjective and we conclude that
$$\varphi_j := \varphi|_{X_j}\colon X_j\rightarrow \overline{\varphi(X_j)}=:Y_j\subseteq Y$$
is dominant for every irreducible component $X_j$. Applying~\cite{RoMe2009}*{Thm 1.1}, we obtain an open subset $U\subseteq Y_j$ such that for all $y \in U$ we have
$$
\dim(\varphi_j^{-1}(y)) 
= \dim(X_j)-\dim(Y_j) 
= \dim(X_j) - \dim(Y) + k
$$
with $k\geq 0$. Note that the latter equality holds as $Y_j\subseteq Y$ is a closed subvariety. 
Now for any $y\in \CC^{r+1}$ we have 
$$\varphi^{-1}(y) = V(T_0^{l_0}-y_0,\ldots,T_r^{l_r}-y_r)\subseteq \CC^{n+m}$$ and thus $\dim(\varphi^{-1}(y)) = n+m-(r+1)$ holds.
We conclude
$$\dim(X_j)-\dim(Y)+k = \dim(\varphi_j^{-1}(y)) \leq \dim(\varphi^{-1}(y)) = n+m-(r+1)$$
and therefore $\dim(X_j)\leq n+m-s+k$ holds, which gives the assertion.
\end{proof}

\begin{proof}[Proof of Theorem~\ref{thm:arrIntNormal}]
In order to prove this statement it suffices to consider indecomposable algebras $R(A,P_0)$.
Let $\mathcal{B}$ be a basis for $\ker(A)$
and fix $v \in \mathcal{B}$. Then the linear forms
$$f_v:= v_0T_0 + \ldots + v_r T_r \in \CC[T_0, \ldots, T_r]$$
are $\ZZ$-homogeneous with respect to the standard $\ZZ$-grading on $\CC[T_0, \ldots, T_r]$. 
In particular, applying Lemma~\ref{lem:connected} we conclude that $X := \mathrm{V}(g_v; \ v \in \mathcal{B})$ is connected.

We want to use Serre's criterion to show that
$X$ is normal and $\mathrm{I}(X) = \bangle{g_1, \ldots, g_{r-c}}$ holds. In particular, as $X$ is connected this implies that $R(A,P_0)$ is integral.
Assume $A$ to be in reduced row echelon form
as in Remark~\ref{rem:admissible} and set $A' := (a_{ij}')_{i,j} := (a_{c+1}, \ldots, a_r)$. 
Recall that the relations $g_1, \ldots, g_{r-c}$
are of the form $g_{v_1}, \ldots, g_{v_{r-c}}$,
where $v_i$ denotes the $i$-th row of the following block-matrix:
$$
\left[
\begin{array}{c|c}
     (A')^t&-E_{r-c}
\end{array}
\right]
$$
Now, set $\delta_i := \mathrm{grad}(T_i^{l_i})$
and $J_1 :=(a_{ji}'\cdot \delta_i)_{i,j}$.
Then the Jacobian of $g_1, \ldots, g_{r-c}$ is of the form
$$
J 
= 
\left[
\begin{array}{c|c}
J_1 &
\begin{array}{ccc}
-\delta_{c+1}&&\\
&\ddots&\\
&&-\delta_r
\end{array}
\end{array}
\right].
$$
Now assume $J(x)$ is not of full rank.
Then there exist at least two indices
$c+1 \leq i_1 < i_2 \leq r$ such that
$\delta_{i_k}(x) = 0$ holds.
Moreover, as the columns of $A'$ are pairwise linearly independent, we have $\delta_{i_3}(x) = 0$
for at least one more index $0 \leq i_3 \leq c$.
In particular, this implies that there exist
$1 \leq j_k \leq n_{i_k}$ such that
$x_{i_1,j_1} = x_{i_2,j_2} =x_{i_3,j_3} =0$
holds.
We conclude that any $x$ with $J(x)$ not of full rank is contained in one of the finitely many affine subvarieties of $X$ of the following form:
$$
\mathrm{V}(\tilde{f}_1(T_0^{l_0}, \ldots, T_r^{l_r}), \ldots, \tilde{f}_{r-c}(T_0^{l_0}, \ldots, T_r^{l_r}), T_{i_1,j_1}, T_{i_2,j_2},T_{i_3,j_3}),
$$
where $\tilde{f}_i := f_i(\tilde{T}_0,\ldots,\tilde{T}_r)$ with 
$\tilde{T}_{i_k}:=0$ for $k = 1,2,3$
and $\tilde{T}_i := T_i$ else.
We claim that these subvarieties are of codimension at least $2$ in $X$.
By Proposition~\ref{prop:Xpure} it suffices to show that 
$\tilde{Y} := \mathrm{V}(\tilde{f}_1, \ldots, \tilde{f}_{r-c}, T_{i_1}, T_{i_2}, T_{i_3}) \subseteq \CC^{r+1}$
is of codimension at least $r-c+2$. 
For this note that 
$$
Y':=\mathrm{V}(\tilde{f}_i; \ i \notin \left\{i_1 - c,i_2 - c\right\}) \cap \mathrm{V}(T_{i_1}, T_{i_2}, T_{i_3}) \subseteq \CC^{r+1}
$$
is irreducible and of codimension $r-c+1$.
Consider the matrix $B$
arising out of $A'$ by replacing
its $i_3$-th row with a zero row.
Then by construction for $k = 1,2$ we have 
$\tilde{f}_{i_k-c} = f_v$,
where $v =(b_{1,i_k}, \ldots, b_{c+1, i_k}, 0, \ldots, 0) \in \CC^{r+1}.$
In particular
$\tilde{f}_{i_k-c} \in I(Y')$ for $k \in \left\{1,2\right\}$ if any only if
$\tilde{f}_{i_k-c} = 0$.
We conclude that
$\tilde{Y} \subseteq \CC^{r+1}$ is of codimension at most $r-c+1$ 
if and only if $\tilde{f}_{i_1-c} = \tilde{f}_{i_2-c} = 0$ holds. This contradicts the fact that the columns of $A'$ are pairwise linearly independent.

In order to complete the proof we have to show
that the $K_0$-grading
has the desired properties.
By construction, the $K_0$-grading is effective.
Moreover, using Remark~\ref{rem:quasihom}  
we obtain a one parameter subgroup of
$H_0 := \Spec \, \CC[K_0]$ via
$$
\CC^* \rightarrow H_0,
\qquad
t \mapsto (t^{\zeta_{01}}, \ldots, t^{\zeta_{rn_r}}, t, \ldots, t),
$$
where $\zeta_{ij} = \deg(T_{ij})$ is as in (\ref{equ:degTij}) with $w_0 = \ldots = w_r =1$.
As $\zeta_{ij} > 0$ holds for all $0\leq i \leq r,$ and $1 \leq j \leq n_i$ 
we conclude that the grading is pointed.
To obtain factoriality of the $K_0$-grading,
we localize $R(A,P_0)$ by the product over all 
generators $T_{ij}$, $S_k$, and observe that the 
degree zero part of the resulting ring is 
a polynomial ring. Now applying~\cite{Be2012}*{Thm.~1.1}
completes the proof.
\end{proof}

\section{(No) Smooth special arrangement varieties with low Picard number}\label{sec:4}
In this section we investigate smoothness of arrangement varieties. Whereas smooth general arrangement varieties of complexity and Picard number two were classified in~\cite{HaHiWr2019}, we concern ourselves with the case of special arrangement varieties.

\begin{definition}
An {\em honestly special arrangement variety} is a special arrangement variety $X$ with {\em honestly special arrangement Cox ring} $R(A,P)$, i.e.\ 
we have $l_{ij}n_i > 1$ for all $i=0,\ldots,r$ and $j=1,\ldots,n_i$ and the graded ring $R(A,P)$ is
not isomorphic to a Cox ring $R(A',P')$ of a general arrangement variety.
\end{definition}

\begin{remark}
A special arrangement variety is honestly special if and only if it does not admit a torus action that turn it into a general arrangement variety.
\end{remark}

\begin{example}\label{ex:notHonest}
Consider the ring $R(A,P)$ defined by the following data
$$
A:=\left[
\begin{array}{cccc}
1&0&0&1\\
0&1&0&1\\
0&0&1&0
\end{array}
\right],\qquad
P:=\left[
\begin{array}{cccc}
-2&2&0&0\\
-2&0&2&0\\
-2&0&0&2\\
-1&1&1&1
\end{array}
\right].
$$
Then any variety $X(A,P,\Sigma)$ is a special arrangement variety, that is not honestly special:
Consider the matrix
$$A':=\left[
\begin{array}{cccc}
1&0&1\\
0&1&1\\
\end{array}
\right].$$
Then the ring $R(A,P)$ is isomorphic as a graded ring to the ring $R(A',P)$, which in turn is the Cox ring of a complexity one $\TT$-variety.
\end{example}

\begin{theorem}\label{thm:noSmoothOnes}
Let $X$ be a projective honestly special arrangement variety of Picard number at most two. Then $X$ is singular.
\end{theorem}

The rest of this section is dedicated to the proof of Theorem~\ref{thm:noSmoothOnes}. We work in the language of explicit $\TT$-varieties~\cite{HaHiWr2019}*{Sec. 3}
and import the necessary notions and facts directly from~\cite{HaHiWr2019}*{Sec. 5}.

\begin{remark}
Let $X:=X(A,P,\Sigma) \subseteq Z$ be an explicit arrangement variety with Cox ring $R(A,P)$.
We will now describe the local structure of $X$ in terms of torus orbits in $Z$. For this, denote by $T$ the acting torus of $Z$. An {\em $X$-cone} is a cone $\sigma\in\Sigma$ such that the corresponding torus orbit $T\cdot z_\sigma$ intersects $X$ non-trivially, i.e.\
$$
X(\sigma):=X\cap T\cdot z_\sigma\not = \emptyset.
$$
We call the subsets $X(\sigma)\subseteq X$, where $\sigma$ is an $X$-cone, the {\em pieces} of $X$. 
Note that the pieces of $X$ are locally closed and $X$ is the disjoint union of all of them:
$$X=\bigsqcup\limits_{X\text{-cones } \sigma} X(\sigma).$$
\end{remark}

Using Cox quotient construction for toric varieties~\cite{Co1995} or~\cite{ArDeHaLa2015}*{Sec. 2.1.3} as indicated in Construction~\ref{constr:XAPSigma}, we describe the pieces of an explicit arrangement variety in terms of faces of the positive orthant.

\begin{remark}
Let $X:=X(A,P,\Sigma) \subseteq Z$ be an explicit arrangement variety with Cox ring $R(A,P)$.
An {\em $\bar{X}$-face} is a face $\gamma_0\preceq \gamma$ of the positive orthant $\gamma:=\QQ^{n+m}_{\geq 0}$, such that the torus orbit defined by the complementary face $\gamma^*\preceq\gamma$ intersects $\bar{X}$ non-trivially, i.e.\
$$
\CC^{n+m}\supseteq \bar{X}(\gamma_0):=\bar{X}\cap T\cdot z_{\gamma_0^*}\not = \emptyset.
$$
Moreover we call $\gamma_0\preceq\gamma$ an {\em $X$-face}, if it is an $\bar{X}$-face and the projected complementary face $P(\gamma_0^*)$ is contained in $\Sigma$.
Note that $\gamma_0\preceq\gamma$ is an $X$-face if and only if the projected complementary face is an $X$-cone, i.e.\
$$
X(\gamma_0):=X(P(\gamma_0^*))\not = \emptyset.
$$
In particular, if $\gamma_0\preceq \gamma$ is an $X$-face, then $\bar{X}(\gamma_0)$ is mapped on $X(\gamma_0)$ under the quotient map $\hat{X}\rightarrow X$.
\end{remark}

\begin{remark}\label{rem:smoothCrit}
Let $X:=X(A,P,\Sigma)\subseteq Z$ be an explicit arrangement variety. Then $X$ is smooth 
if for every $X$-face $\gamma_0\preceq \gamma$, 
the piece $X(\gamma_0)$ consists of factorial points of $X$ and $\bar{X}(\gamma_0)$ consists of smooth points of $\bar{X}$.
$\QQ$-factoriality and factoriality of points of $X(\gamma_0)$ can be checked using the following criterion, see~\cite{HaHiWr2019}*{Sec. 5}:
\begin{enumerate}
\item $x\in X(\gamma_0)$ is $\QQ$-factorial if and only if the cone $Q(\gamma_0)\subseteq K_\QQ$ is of full dimension.
\item $x\in X(\gamma_0)$ is factorial if and only if the set $Q(\gamma_0\cap\ZZ^{n+m})\subseteq K$ generates $K$ as a group.
\end{enumerate}
\end{remark}

We turn to the proof of Theorem~\ref{thm:noSmoothOnes}.
In a first step we treat the case of Picard number one.

\begin{remark}
Let $X(A,P,\Sigma)$ be an honestly special arrangement variety. Then $R(A,P)$ is defined by at least two relations $g_v$.
\end{remark}

\begin{remark}\label{rem:relations}
Let $X:=X(A,P,\Sigma)$ be a projective honestly special arrangement variety.
After suitably renumbering we may assume that for the defining relations
\begin{equation}
g_t = \lambda_{t0}T_0^{l_0}+\ldots+\lambda_{tc}T_c^{l_c}+T_{t+c}^{l_{t+c}},\ \text{ where } 1\leq t\leq r-c
\end{equation}
of $R(A,P)$ there exists an index $k\in\{0,1,\ldots,c\}$ such that $\lambda_{1k} = 0$ and $\lambda_{2k}\not = 0$ holds.
In particular the face $$\gamma_0:=\cone(e_{k1},\lambda_{tk}e_{(t+c)1};\ 2\leq t\leq r-2)\preccurlyeq\gamma$$
is an $\overline{X}$-face. Moreover the corresponding stratum in $\overline{X}$ is singular due to the number of relations defining $R(A,P)$.
\end{remark}

\begin{remark}\label{rem:Pic1}
Let $X:=X(A,P,\Sigma)$ be a projective arrangement variety of Picard number one. Then every $\overline{X}$-face $\{0\}\not = \gamma_0\preccurlyeq\gamma$ is an $X$-face.
\end{remark}

\begin{proof}[Proof of Theorem~\ref{thm:noSmoothOnes} for $\varrho(X) = 1$]
Assume there is a smooth explicit honestly special arrangement variety $X := X(A,P,\Sigma)$ as in Theorem~\ref{thm:noSmoothOnes} with $\varrho(X) = 1$. Then using Remark~\ref{rem:relations} we obtain a $\overline{X}$-face with singular stratum in $\overline{X}$ which is an $X$-face due to Remark~\ref{rem:Pic1}. This contradicts smoothness of $X$ due to Remark~\ref{rem:smoothCrit}.
\end{proof}

We turn to Picard number two. In a first step we import the description of the several cones of divisor classes inside the rational divisor class group $\Cl(X)_\QQ:=\Cl(X)\otimes\QQ$ of an arrangement variety. Then, specializing to Picard number two, we adapt techniques for treating these varieties from~\cites{FaHaNi2018,HaHiWr2019}. Using these techniques, we obtain in Lemma~\ref{lem:tool} first constraints on the defining data $A,P,\Sigma$. Then we go on proving Theorem~\ref{thm:noSmoothOnes} for $\varrho(X)=2$.

\begin{remark}
Let $X(A,P,\Sigma)\subseteq Z$ be an explicit arrangement variety. Then the cones of effective, movable, semiample and the (open) cone of ample divisor classes inside $K_\QQ:=K\otimes\QQ$ are given as
$$ 
\Eff(X) 
\ = \ 
Q(\gamma), 
\qquad\qquad
\Mov(X)
\ = \ 
\bigcap_{\gamma_0 \preccurlyeq \gamma \text{ facet}}
Q(\gamma_0),
$$
$$ 
\SAmple(X)
\ = \ 
\bigcap_{X\text{-faces } \gamma_0} Q(\gamma_0),
\qquad\qquad
\Ample(X)
\ = \ 
\bigcap_{X\text{-faces } \gamma_0} Q(\gamma_0)^\circ.
$$
\end{remark}

\begin{remark}\label{rem:divClass}
Let $X := X(A,P,\Sigma)$ 
be an explicit arrangement variety
with divisor class group $K=\Cl(X)$ 
of rank two.
Then, inside the rational divisor class group $\Cl(X)_\QQ=\QQ^2$,  the effective cone of $X$ is of 
dimension two and decomposes as
$$
\Eff(X) 
\ = \ 
\tp \cup \tx \cup \tm,
$$
where $\tx \subseteq \Eff(X)$ is the 
ample cone, $\tp$, $\tm$ are closed cones 
not intersecting $\tx$ and 
$\tp \cap \tm$ consists of the origin:
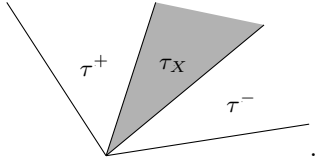
\vspace{10pt}
\begin{center}
\begin{tikzpicture}[scale=0.6]
    \path[fill=gray!60!] (0,0)--(3.5,2.9)--(1.1,3.4)--(0,0);
    \path[fill, color=black] (1.5,2) circle (0.0ex)  node[]{\small{$\tx$}};
    \path[fill, color=black] (-0.25,2) circle (0.0ex)  node[]{\small{$\tp$}};
    \draw (0,0)--(1.1,3.4);
    \draw (0,0) --(-2.2,3.4);
    \path[fill, color=black] (3,1.2) circle (0.0ex)  node[]{\small{$\tm$}};
    \draw (0,0)  -- (3.5,2.9);
    \draw (0,0)  -- (4.5,0.7);
  \end{tikzpicture}.
\end{center}
Due to $\tx \subseteq \Mov(X)$,
each of the cones $\tp$ and~$\tm$ 
contains at least two of the rational weights
$$
w_{ij}:= (x_{ij},y_{ij}):= 
\deg_\QQ(T_{ij}),
\qquad
w_k := (x_k,y_k):=
\deg_\QQ(S_k).
$$
Moreover, for every $\overline{X}$-face
$\{0\} \ne \gamma_0 \preccurlyeq \gamma$
precisely one of the following 
inclusions holds:
$$
Q(\gamma_0) \ \subseteq \ \tp,
\qquad
\tx \ \subseteq \ Q(\gamma_0)^\circ,
\qquad
Q(\gamma_0) \ \subseteq \ \tm.
$$
The $X$-faces are precisely those $\overline{X}$-faces $\gamma_0\preceq \gamma$ with $\tx  \subseteq  Q(\gamma_0)^\circ$.
\end{remark}

\begin{remark}\label{rem:normForm}
In the situation of Remark~\ref{rem:divClass} consider 
a positively oriented pair $w,w'\in\QQ^2$.
If, for instance, $w \in \tau^-$ and $w' \in \tau^+$ hold,
then $\det(w,w')$ is positive.
Moreover, if the variety $X$ is smooth and $w,w'$ 
are the weights stemming from
a two-dimensional $X$-face 
$\gamma_0 \preccurlyeq \gamma$, then we have
$\det(w,w')=1$ due to Remark~\ref{rem:smoothCrit}[(ii)].
In this case, we can achieve 
$$ 
w \ = \ (1,0), 
\qquad
\qquad
w' \ = \ (0,1)
$$
by a suitable unimodular coordinate change on 
$\ZZ^2\subseteq\QQ^2$. Then $w'' = (x'',1)$ holds whenever $w,w''$ are the weights 
stemming from a two-dimensional $X$-face and, similary, 
$w'' = (1,y'')$ holds whenever $w'',w'$ are these weights.
\end{remark}

We call an explicit arrangement variety $X:=X(A,P,\Sigma)$ {\em quasismooth}, if for every $X$-face, the corresponding stratum in $\bar{X}$ is smooth.

\begin{lemma}\label{lem:tool}
Let $X:=X(A,P,\Sigma) \subseteq Z$ be a $\QQ$-factorial quasismooth projective
honestly special arrangement variety of
Picard number two.
Then the following assertions hold:
\begin{enumerate}
    \item If $m>0$ holds then all weights $w_k$ lie either in $\tau^+$ or in $\tau^-$. 
    \item If $n_i \geq 2$ holds for at least one index $0\leq i\leq r$ then $m = 0$ holds.
    \item $n_i \leq 2$ holds for all $0 \leq i \leq r$.
\end{enumerate}
\end{lemma}
\begin{proof}
We prove (i). 
Let $m \geq 2$. As $\cone(e_k)$ is an $\overline{X}$-face for $1 \leq k \leq m$, Remark~\ref{rem:normForm} implies
$w_k\not \in \tau_X$ due to $\QQ$-factoriality of $X$.
So assume we have $w_{k_1}\in\tau^+$ and $w_{k_2}\in\tau^-$. Then $\cone(e_{k_1},e_{k_2})$ is an $X$-face with singular stratum in $\overline{X}$; a contradiction to quasismoothness of $X$.

We prove (ii).
Let $m>0$. Then we may assume that $w_k\in\tau^+$ holds for all $1\leq k\leq m$. Assume there exists a weight $w_{ij}\in\tau^-$ 
with $n_i \geq 2$. Then $\cone(e_1,e_{i1})$ is an $X$-face with singular $\overline{X}$-stratum, which contradicts quasismoothness of $X$. Thus
$w_{ij} \in \tau^+$ holds for all $i$ with $n_i \geq 2$. Due to homogeneity 
of the relations we conclude $w_{ij} \in \tau^+$ for all $i$ with $n_i=1$ and there are no weights left to lie in $\tau^-$; a contradiction due to Remark~\ref{rem:normForm}.

We prove (iii). 
Assume there exists an index $i$ with $n_i \geq 3$. Then after suitably renumbering we may assume $i=0$. We claim
that all $w_{0j}$ lie either in $\tau^+$ or in $\tau^-$.
Assume that this is not true and $w_{01}\in \tau^+$ and $w_{02}\in\tau^-$ holds.
Then the cone $\cone(e_{01},e_{02})$ is an $X$-face with singular $\overline{X}$-stratum as there are at least two relations defining $R(A,P)$; a contradiction to quasismoothness of $X$.
Thus we may assume $w_{0j} \in \tau^+$ for $j = 1,2,3$ and homogeneity of the relations implies that all weights $w_{kj}$ with $n_k =1$ or $n_k \geq 3$ lie in $\tau^+$. Using Part (ii) and Remark~\ref{rem:normForm} we conclude that there exist at least two indices $i_1,i_2$ with $n_{i_1} = n_{i_2}=2$ and after suitably renumbering we may assume that $w_{i_1 1},w_{i_2 1} \in \tau^-$ holds.
In particular, the $\overline{X}$-face $\cone(e_{i_1 1}, e_{0 1})$ is an $X$-face. As we have at least two relations defining $R(A,P)$, the corresponding stratum in $\overline{X}$ is singular; a contradiction to quasismoothness of~$X$.
\end{proof}

\begin{proof}[Proof of Theorem~\ref{thm:noSmoothOnes} for $\varrho(X) =2$]
We show that the existence of a smooth variety $X(A,P, \Sigma)$ as in the theorem leads to a contradiction in all possible cases.

Assume $n_i = 1$ holds for all $0\leq i\leq r$. Then due to homogeneity of the relations we may assume that all weights $w_{i1}$ lie in $\tau^+$. Thus due to Remark~\ref{rem:normForm} there exist at least two weights $w_1,w_2\in\tau^-$.
Due to Remark~\ref{rem:relations}
there exists an index $k\in\{0,1,\ldots,c\}$ such that
$\cone(e_1,e_{k1},\lambda_{tk}e_{(t+c)1};\ 2\leq t\leq r-2)$
is an $X$-face with singular 
stratum in $\overline{X}$; a contradiction to smoothness of~$X$.

Now, due to Lemma~\ref{lem:tool} we may assume that
$2 = n_0 \geq \ldots \geq n_r\geq 1$ and $m=0$ holds.
Due to homogeneity of the relations we may assume that all weights $w_{i1}$ with $n_i=1$ lie in $\tau^+$
and thus due to Remark~\ref{rem:normForm}
$n_0=n_1=2$ holds with $w_{01},w_{11}\in \tau^+$ and $w_{02},w_{12}\in\tau^-$.
Considering the $X$-faces $\cone(e_{01},e_{12}),\cone(e_{02},e_{11})$, quasismoothness of $X$ implies $l_{01}=l_{02}=l_{11}=l_{12} = 1$.
After suitably renumbering we may moreover assume $w_{11}\in\cone(w_{01},w_{02})$. And thus applying Remark~\ref{rem:normForm} to the $X$-face $\cone(e_{01},e_{12})$ 
turns the degree matrix of the rational degrees $Q$ into the shape
$$
Q
=
\left[ 
\begin{array}{cc|cc|c}
1 & x_{02} & x_{11} & 0 & \ldots
\\
0 & y_{02} & y_{11}&1 & \ldots
\end{array}
\right],
$$
where $x_{11}, y_{11} \geq 0$.
Applying Remark~\ref{rem:normForm} to the $X$-face $\cone(e_{11},e_{02})$ we obtain
$1 = \det(w_{11},w_{02}) = x_{11}y_{02} -x_{02}y_{11}$.
Using homogeneity of the relations we obtain
$$
y_{02}=l_{02}y_{02} = l_{11}y_{11}+l_{12} = y_{11}+1, \quad
1+x_{02}=l_{01}+l_{02}x_{02}=l_{11}x_{11} = x_{11}.
$$
This implies $x_{02} = -y_{11}$ and $y_{02} = 2 - x_{11}$, hence $1-y_{11} = x_{11}\geq 0$ and thus $0\leq y_{11} \leq 1$.
Assume $y_{11} = 0$. This turns the degree matrix of the rational degrees $Q$ into the shape
$$
Q
=
\left[ 
\begin{array}{cc|cc|c}
1 & 0 & 1 & 0 & \ldots
\\
0 & 1 & 0&1 & \ldots
\end{array}
\right].
$$
In particular, the rational degree of the relations is $(1,1)$. This implies $n_i\not = 1$ for all $0\leq i \leq r$ due to the honesty of $R(A,P)$. 
Assume there exists an index $k$ with $w_{k1},w_{k2}\in\tau^+$ then $\cone(e_{k1},e_{02}),\cone(e_{k2},e_{02})$ are $X$-faces and applying Remark~\ref{rem:normForm}
gives $w_{k1} = (1,y_{k1})$ and $w_{k2} = (1,y_{k2})$ in contradiction to
homogeneity of the relations.
Similar arguments hold for $w_{k1},w_{k2}\in\tau^-$.
Thus we may assume $w_{i1}\in\tau^+$ and $w_{i2}\in\tau^-$ for all $0\leq i\leq r$.
Due to Remark~\ref{rem:relations}
there exists an index $k\in\{0,1,\ldots,c\}$ such that $\cone(e_{k1},e_{(c+2)2})$ is an $X$-face with singular stratum in $\overline{X}$. This contradicts smoothness of $X$.
Thus we may assume $y_{11}=1$ this gives
$w_{11} = (0,1) = w_{12}$ which in turn is a contradiction to $w_{11}\in\tau^+$ and $w_{12}\in\tau^-$.
\end{proof}

\section{Toric ambient resolutions of singularities}\label{sec:5}
The purpose of this section is to prove that explicit arrangement varieties allow a toric ambient resolution of singularites:
Let $X(A,P, \Sigma) \subseteq Z$ be an explicit arrangement variety, denote by $T$ the acting torus of $Z$
and let $\varphi \colon Z' \rightarrow Z$ be a birational toric morphism. We call $\varphi$ an \emph{ambient resolution of singularities} if it maps the \emph{proper transform} $X'$, i.e.\ the closure in $Z'$ of the preimage of $X\cap T$ under $\varphi$, properly onto $X$ and $X'$ is smooth.
The existence of an ambient resolution of singularities enables us to perform the desingularization of $X$ in a purely combinatorial manner.

\begin{theorem}\label{thm:locallytoric}
Let $X := X(A,P,\Sigma) \subseteq Z$ be an explicit arrangement variety. Then $X \subseteq Z$ admits a toric ambient resolution of singularities. 
\end{theorem}

In order to prove the above result we will use 
a two-step procedure for resolving singularities adapted from \cite{BeHaHuNi2016}.
In the first step we make use of tropical methods.
Fixing a suitable quasifan structure on the tropical variety enables us to show that
after the first resolution step
the variety is locally toric in a strong sense; see Definition~\ref{def:locallyToric}.
Then applying~\cite{HiWr2018}*{Thm 4.5. and Prop. 2.6}
we obtain the existence of a toric ambient resolution of $X \subseteq Z$.

Let $X:=X(A,P, \Sigma) \subseteq Z$ be an explicit arrangement variety and let $\trop(X)$ be its tropical variety endowed with a fixed quasifan-structure.
We call $X \subseteq Z$ \emph{weakly tropical} if $\Sigma$ is supported on $\trop(X)$. 
Set  $\Sigma' := \Sigma \sqcap \trop(X)$
and let $\varphi \colon Z' \rightarrow Z$ be the toric morphism corresponding to the refinement of fans $\Sigma' \rightarrow \Sigma$. Let $X'$ denote the proper transform of $X$ with respect to $\varphi$.  
We call $\varphi \colon Z' \rightarrow Z$ the \emph{weakly tropical resolution} of $X$.
Note that the weakly tropical resolution of $X$ depends on 
the choice of the quasifan structure fixed on $\trop(X)$.
Let us look at two possible choices of quasifan-structures, compare~\cite{MaSt2015}*{Chap. 4}.

\begin{remark}\label{rem:tropical}
Let $X(A,P, \Sigma) \subseteq Z$ be an explicit arrangement variety. Then, due to Construction~\ref{constr:MOQ}, the matrix $A$ gives rise to a linear embedding $\PP_c \subseteq \PP_r$.
Moreover, the projection $P_1 \colon \QQ^{r+s} \rightarrow \QQ^r$ onto the first $r$ coordinates maps $\trop(X)$ onto $\trop(Y)$ and we obtain
$$|\trop(X)| = |\trop(\PP_c \cap \TT^r)| \times \QQ^s.$$
In the following we will construct a fan structure on $\trop(Y)$ and will endow $\trop(X)$ with the corresponding quasifan structure, i.e.
$$\trop(X) = \left\{P_1^{-1}(\lambda); \ \lambda \in \trop(Y)\right\}.$$
Let $A$ be as above and denote by $\mathcal{A}$ the set of columns of $A$. 
The \emph{lattice of flats} $\mathcal{L}(A)$ is the partially ordered set of all 
subspaces of $\CC^{r+1}$ spanned by subsets of $\mathcal{A}$. Note that all maximal chains in $\mathcal{L}(A)$ have length $c+1$.
For any $S \in \mathcal{L}(A)$ denote by 
$I(S) \subseteq \left\{0, \ldots, r \right\}$ the indices 
with $a_i \in S$ and set
$e_S:= \sum_{i \in I(S)} e_i$, where $e_0 := - \sum_{i=1}^r e_i$. 
For any maximal chain $S_1 \subseteq S_2 \subseteq \ldots \subseteq S_c \subseteq \CC^{r+1}$ 
we define a cone
$\cone(e_{S_1}, \ldots, e_{S_c})$
and denote with $\Delta(\mathcal{A})$
the fan having these cones as maximal ones. 
Then due to~\cite{MaSt2015}*{Thm.\ 4.3.7}
this defines a fan structure on the tropical variety
$\trop(\PP_c)$.

Note that the tropical variety of a variety $Y \subseteq \TT^r$ defined by linear relations can be endowed with a unique coarsest fan structure, the so called Bergman-fan, see~\cite{MaSt2015}*{Chap. 4}. 
\end{remark}

\begin{construction}
\label{constr:decomp}
Let $X\subseteq Z$ be weakly tropical and let $\sigma \in \Sigma$ be any cone. 
Choose a maximal cone $\tau \in \trop(X)$ with $\sigma \preceq \tau$, set 
$N(\tau):= N\cap \lin_\QQ(\tau)$ 
and fix a decomposition $N = N(\tau) \oplus \tilde{N}$. Accordingly, we obtain a product decomposition
$$Z_\sigma \cong U(\sigma) \times \tilde\TT,$$
where $U(\sigma) := U(\sigma, \tau)$ is the affine toric variety corresponding to the cone $\sigma$ in the lattice $N(\tau)$
and $\tilde\TT $ is a torus.
\end{construction}

\begin{definition}\label{def:locallyToric}
Let $X\subseteq Z$ be weakly tropical.
Call $X\subseteq Z$ {\em semi-locally toric} if
for every maximal cone $\sigma \in \Sigma$ there exists a projection $\pi_{\sigma}$ as in Construction~\ref{constr:decomp} that maps $X_{\sigma}:=X \cap Z_{\sigma}$ isomorphically onto its image $\pi_{\sigma}(X_\sigma)$ and the latter is an open subvariety of $U(\sigma)$.
\end{definition}

\begin{lemma}\label{lem:locallyToric}
Let $X := X(A,P,\Sigma) \subseteq Z$ be an explicit arrangement variety. Then $X \subseteq Z$ admits a  semi-locally toric weakly tropical resolution.
\end{lemma}

\begin{proof}
Due to~\cite{HiWr2018}*{Proof of Thm 4.5}
to show that the embedding $\PP_c \subseteq \PP_r$ defined by the matrix $A$ admits a semi-locally weakly tropical resolution. For our purposes we will endow $\trop(\PP_c)$ with the fan structure $\Delta(\mathcal{A})$ defined in Remark~\ref{rem:tropical}. 
Denote the corresponding fan of $\PP_r$ with $\Delta$.
Then $\trop(\PP_c) \sqcap \Delta = \trop(\PP_c)$ holds.
Let $\tau \in \trop(Y)$ be a maximal cone. Then $\tau \subseteq \delta$ holds for a cone $\delta \in \Delta$
and after a coordinate change we may assume $\delta = \QQ_{\geq 0}^r$.
Denote by $v_1, \ldots, v_c$ the primitive generators of the rays of $\tau$.
Then by definition of the fan structure on $\trop(Y)$ 
and after suitably renumbering we achieve
$$v_i = \sum_{j =1}^{k_i} e_j, 
\quad \text{where} \quad
1 = k_1 < k_2 < \ldots < k_c.$$
We complement the set $v_1, \ldots, v_c$ to a lattice basis of $\ZZ^{r}$ by successively adding canonical basis vectors in the following way:
Whenever $k_{i+1} > k_{i}+1$
we add $e_{k_i + 2}, \ldots, e_{k_{i+1}}$.
Moreover, we add the vectors $e_{k_c +1}, \ldots e_r$.
This gives rise to a decomposition
$\ZZ^{r}=:N = N(\tau) \oplus \tilde{N}$
as in Construction~\ref{constr:decomp}
and we obtain an isomorphism
$Z_\tau \cong U(\tau) \times \TT^{r-c} \cong \CC^{c} \times \TT^{r-c}$. On the torus this isomorphism is given by the homomorphism
$\varphi_B \colon \TT^r \rightarrow \TT^r$
defined by the matrix $B$ whose columns are the above lattice basis. 

Let $I := \bangle{f_1, \ldots ,f_{r-c}}$
be the ideal corresponding to $\TT^r \cap Y$.
For any $f \in I\cap \CC[T_1, \ldots, T_r]$ denote by $\tilde{f}$ the push-down of $f$ with respect to $\varphi_B$,
i.e., the unique $(\varphi_B)_*(f) \in \CC[T_1, \ldots T_r]$ without monomial factors such that $T^\mu\varphi_B^*((\varphi_B)_*(f))  =f$
for a $\mu \in \ZZ_{\geq0}^r$.
Then by construction we have 
$$
Y_\tau \cong \tilde{Y}_\tau := \V(\tilde{f}; \ f \in I\cap \CC[T_1, \ldots, T_r]) \subseteq \CC^c \times \TT^{r-c} \cong Z_\tau,
$$
where the isomorphism on the left hand side is the restriction of the isomorphism on the right hand side.
In order to complete the proof we need to show that the
restriction of the projection onto the first $c$ coordinates to $\tilde{Y}_\tau$ is an isomorphism onto its image and the latter is an open subset of $\CC^c$. 
We show this by proving that for any $k > c$ there exists a push-down $\tilde{f}$ of an equation $f \in I$ 
of the form
$\tilde{f} = h + \lambda_k T_k$,
where $h \in \CC[T_1, \ldots, T_{k-1}]$ and $\lambda_k \neq 0$. 

Our proof is by induction.
Let $k = c+1$.
Recall that the $(c+1)$-th column of $B$ is a canonical basis vector $e_i$ for some $1 \leq i \leq r$.
We distinguish between the following two cases:

\vspace{1mm}
\noindent
\emph{Case 1:}
There exists a ray generator $v_{j(i)}$ with
$k_{j(i)} > i$. Let $j(i)$ be minimal with this property.
By construction of $A$ we have $i = k_{j(i)-1}+2$. 
Moreover, due to the construction of the fan structure on $\trop(Y)$ the columns $a_1, \ldots, a_i$ of $A$ are linearly dependant and there exists 
a push-down $\tilde{f} = h +  \lambda_k T_k$
as claimed. 

\vspace{1mm}
\noindent
\emph{Case 2:}
We have $i = k_c+1$. Then 
the columns
$a_1, \ldots, a_{k_c}, a_i, a_0$ are linearly dependant
as $\tau$ was chosen maximal.
In particular there exists a push down
$\tilde{f} = h +  \lambda_0 + \lambda_i T_i$
with $\lambda_i \neq 0$ and $h \in \CC[T_1, \ldots, T_{k_c}]$.

\vspace{3mm}
Now assume we have proven the above 
for all $c < k \leq n$.
Consider the case $k = n+1$.
As above the $(n+1)$-th column of $B$ is a canonical basis vector $e_i$ for some $1 \leq i \leq r$
and we follow the same lines as in the induction basis:

\vspace{1mm}
\noindent
\emph{Case 1:}
There exists a ray generator $v_{j(i)}$ with
$k_{j(i)} > i$. Let $j(i)$ be minimal with this property,
i.e. we have $k_{j(i) - 1} < i < k_{j(i)}$
and there exists $\alpha \geq 2$ with
$i = k_{j(i)-1}+\alpha$.
We conclude that the columns $a_1, \ldots, a_i$ of $A$ are linearly dependant and there exists 
a push-down $\tilde{f} = h +  \lambda_k T_k$
with $h \in \CC[T_1, \ldots, T_{k-1}]$
as claimed. 

\vspace{1mm}
\noindent
\emph{Case 2:} 
We have $i > k_c$ and the existence of a push-down $\tilde{f}$ follows with exactly the same arguments as in the induction basis.
\end{proof}

\begin{proof}[Proof of Theorem~\ref{thm:locallytoric}]
Due to Lemma~\ref{lem:locallyToric}
the embedding $X\subseteq Z$
admits a semi-locally weakly tropical resolution.
Therefore the assertion follows using
\cite{HiWr2018}*{Prop. 2.6}.
\end{proof}

\section{The anticanonical complex for arrangement varieties}\label{sec:6}
In this section we investigate the anticanonical complex for arrangement varieties. Note that for general arrangement varieties it was described in~\cite{HiWr2018}.
In the first part, we recall and adapt the general construction of  anticanonical complexes to the arrangement case. In order to get an explicit description, see Proposition~\ref{prop:disc}, we will restrict ourselves to the case, where we have access to canonical divisors of the weakly tropical resolution via the theory of Cox rings, see Definition~\ref{def:tropres}.

We start by recalling the necessary facts:
Let $X$ be a normal $\QQ$-Gorenstein variety
and $\varphi \colon X' \rightarrow X$ 
a proper birational morphism with $X'$ normal. 
Then, given a canonical divisor $k_{X'}$ on $X'$
we have the ramification formula
$$
k_{X'} - \varphi^*\varphi_*k_X' = \sum a_E E
$$
where the sum on the right hand side runs over all 
prime divisors $E$ in the exceptional locus of $\varphi$.
The number 
$\mathrm{dics}_X(E) = a_E$ is called the \emph{discrepancy} of
$X$ with respect to $E$.

Now let $X:=X(A,P, \Sigma) \subseteq Z$ be an explicit arrangement variety. Then due to Lemma~\ref{lem:locallyToric} it admits a semi-locally toric weakly tropical resolution $Z'\rightarrow Z$ where $Z'$ is defined via a fan $\Sigma'$ as in the previous section. 
Applying~\cite{HiWr2018}*{Prop.3.6}, we can construct an anticanonical complex for $X$ in the following way:
For every ray $\varrho \subseteq |\Sigma'|$ there exists by assumption a proper toric morphism 
$\psi \colon Z''\rightarrow Z'$ where $Z''$ is a toric variety defined by a fan $\Sigma''$ with $\varrho \in \Sigma''$. This induces a proper birational morphism $X''\rightarrow X$, where $X''$ is the closure of the preimage of $X\cap T$ under the composition $Z''\rightarrow Z'\rightarrow Z$.
Denote by $D_{Z''}^{\varrho}$ the toric prime divisor 
corresponding to $\varrho$ and set 
$$
a_{\varrho} := \mathrm{disc}_X(D_{Z''}^{\varrho}|_{X''}).
$$
Let $v_\varrho$ denote the primitive ray generator of $\varrho$ and for $a_{\varrho} > -1$ set $v_\varrho' := \frac{1}{a_{\varrho} +1 }v_{\varrho}$.
Set
$$
\mathcal{A} := \bigcup_{\small \varrho \ \! \subseteq \ \! |\Sigma'|} \mathcal{A}_\varrho, \qquad
\mathcal{A}_{\varrho} 
:=
\begin{cases}
\conv(0, v_\varrho')
,& \text{ if } a_{\varrho} > -1
\\
\varrho,& \text{ else. } 
\end{cases}
$$
Then due to~\cite{HiWr2018}, $\mathcal{A}$ admits the structure of a polyhedral complex by setting
$$
\mathcal{A} = \bigcup_{\sigma' \in \Sigma'} (\mathcal{A} \cap \sigma')
$$
and we call $\mathcal{A}$ the \emph{anticanonical complex}
of $X \subseteq Z$.
By construction it encodes the singularity type of $X$ in the following way:

\begin{enumerate}
    \item[(i)] $X$ has at most log terminal singularities if and only if $\mathcal{A}$ is bounded. 
\item[(ii)]
$X$ has at most canonical singularities
if and only if $0$ is the only 
lattice point in the relative interior of 
$\mathcal{A}$.
\item[(iii)]
$X$ has at most terminal singularities
if and only if $0$ and the primitive generators 
of the rays of the fan of $Z$ are the only 
lattice points of~$\mathcal{A}$.
\end{enumerate}

We turn to the construction of candidates for the Cox ring of the weakly tropical resolution of an explicit arrangement variety $X:=X(A,P,\Sigma)$. 
We use the concept of toric ambient modifications presented in~\cite{HaKeLa2016}. Let $\trop(X)$ be endowed with a fixed quasifan structure and consider the toric morphism $Z_{\Sigma'} \to Z_\Sigma$ defined via the subdivision $\Sigma' = \trop(X) \sqcap \Sigma \to \Sigma$ of fans.
Denote by $P$ and $P'$ the matrices whose columns are the primitive ray generators of $\Sigma$ and $\Sigma'$.
Then the corresponding maps $P \colon \ZZ^r \to \ZZ^n$ 
and $P' \colon \ZZ^{r'} \to \ZZ^n$ 
define homomorphisms of tori
$$ 
\xymatrix{
{\TT^{r'}}
\ar[r]^{p'}
&
{\TT^{n}}
&
{\TT^{r}}
\ar[l]_{p}
}.
$$
Let $g_i \in \CC[T_1, \ldots, T_r]$
be one of the defining polynomials of $\mathcal{R}(X)= R(A,P)$. 
The push-down of $g_i$ is the unique
$p_*(g_i) \in \CC[T_1,\ldots,T_n]$ without monomial factors 
such that $T^{\mu}p^*(p_*(g_i))= g_i$ holds for some 
Laurent monomial $T^\mu \in \CC[T_1^{\pm 1}, \ldots, T_r^{\pm}]$.
The \emph{shift} of $g_i$ is the unique
$g_i' \in \CC[T_1, \ldots, T_{r'}]$ 
without monomial factors satisfying $p'_*(g_i') = p_*(g_i)$.

\begin{definition}
\label{def:tropres}
Let $X(A,P,\Sigma) \subseteq Z$
be an explicit arrangement variety
with Cox ring 
$$
\mathcal{R}(X) 
\ = \
\CC[T_{\varrho}; \; \varrho \in \Sigma^{(1)}] / \bangle{g_1,\ldots,g_s}.
$$
We call the weakly tropical resolution $X' \to X$ 
arising from a subdivision
$\Sigma \sqcap \trop(X) \to \Sigma$ 
\emph{explicit}
if $X'$ has a complete intersection Cox ring
defined by the shifts $g_i'$ of $g_i$:
$$
\mathcal{R}(X') 
\ = \ 
\CC[T_{\varrho'}; \; \varrho' \in \Sigma'^{(1)}] /\bangle{g'_1,\ldots,g'_s}.
$$ 
\end{definition}

From now on let $\trop(X)$ be endowed with any coarsening of
the quasifan structure defined in Remark~\ref{rem:tropical}. 
In this situation, if the weakly tropical resolution of $X$ is semi-locally toric, it suffices to compute the discrepancies along the divisors corresponding to the rays of $\Sigma'$ to describe the whole anticanonical complex. 
This motivates the subsequent study of the rays of $\Sigma'$.
We work in the notation of Definition~\ref{def:Pcone} and 
denote by
$e_{ij}$ resp. $e_k$ the canonical basis vectors of $\QQ^{n+m}$. We set
$$
v_{ij}:=P(e_{ij}),\qquad v_k:=P(e_k).
$$
Moreover for a fan $\Sigma$ we denote by $\Sigma^{(1)}$ its set of rays.

\begin{definition}\label{def:elemCones}
Let $X(A,P,\Sigma)\subseteq Z$ be an explicit arrangement variety and $\sigma\subseteq\QQ^{r+s}$ be a $P$-cone of special or big type. We call $\sigma$ {\em elementary}, if the following statements hold:
\begin{enumerate}
    \item For all $0\leq i\leq r$ there exists at most one index $1\leq j_i\leq n_i$ such that $v_{ij_i}$ is a primitive ray generator of $\sigma$.
    \item There is a 
    ray $\varrho$
    in $\sigma \sqcap \trop(X)$ with $\varrho \cap \sigma^\circ \neq \emptyset.$
\end{enumerate}
\end{definition}

\begin{construction}
Let $X(A,P,\Sigma)\subseteq Z$ be an explicit arrangement variety and let $\sigma\subseteq\QQ^{r+s}$ be an elementary $P$-cone.
Denote by $I$ the set of indices $i$ such that $v_{ij_i}$ is a primitive ray generator of $\sigma$ and define
$$
\ell_{\sigma,i}:=
\frac{\prod_{k\in I}l_{kj_k}}
{l_{ij_i}} \text{ for } i\in I,
\qquad
v_\sigma:=\sum_{i\in I}\ell_{\sigma,i}v_{ij_i},
\qquad
\varrho_\sigma:=\QQ_{\geq 0}\cdot v_\sigma.
$$
\end{construction}

\begin{proposition}\label{prop:rays}
Let $X:=X(A,P,\Sigma) \subseteq Z$ be an explicit arrangement variety. Then the set of rays of $\Sigma\sqcap\trop(X)$ is given as
$$
(\Sigma\sqcap\trop(X))^{(1)} =
\Sigma^{(1)}\cup\{\varrho_\sigma;\ \sigma\in\Sigma\text{ is elementary}\}.
$$
\end{proposition}

Using the above result and applying the methods developed in~\cite{HiWr2018}
we obtain the following description of the discrepancies 
along the divisors corresponding to the rays of $\Sigma'$ which leads to a full description of the anticanonical complex.

\begin{proposition}
\label{prop:disc}
Let $X:=X(A,P, \Sigma)\subseteq Z$ be a $\QQ$-Gorenstein explicit
arrangement variety
admitting a semi-locally toric
explicit weakly tropical resolution 
$Z'\rightarrow Z$ 
and let 
$\sigma = \cone(v_{ij_i}; \ i \in I) \in \Sigma$ 
be an elementary cone. 
Then the following statements hold:
\begin{enumerate}
\item
The discrepancy along the prime divisor of $X'\subseteq Z'$ 
corresponding to $\varrho_\sigma$ equals 
$c_{\sigma}^{-1}\ell_{\sigma}-1$,
where
$$
\ell_{\sigma}
:=
\sum_{i \in I} \ell_{\sigma,i} - k \cdot \prod_{i \in I} l_{ij_i}
$$
and $k$ is the number of the defining equations 
$g_i$ of the Cox ring $R(A,P)$ of $X$ with $g_i(x) = 0$
for all $x \in \V(T_{ij_i};\  i \in I)$.
\item
The ray $\varrho_\sigma$ is not contained in
the anticanonical complex $\mathcal{A}$, 
if and only if $\ell_\sigma > 0$ holds; 
in this case, $\varrho_\sigma$ leaves $\mathcal{A}$ at
$v_\sigma' = \ell_{\sigma}^{-1} v_\sigma$.
\end{enumerate}
\end{proposition}

The rest of this section is dedicated to the proofs of 
Propositions~\ref{prop:rays} and~\ref{prop:disc}.
In the following let $X(A,P, \Sigma) \subseteq Z$ be an explicit special arrangement variety
with an $(c+1) \times (r+1)$-matrix $A$. Let 
$P_1 \colon \QQ^{r+s} \rightarrow \QQ^r$
denote the projection onto the first $r$ coordinates
and
$\Delta:= \Sigma_{\PP_r}$ the fan corresponding to $\PP_r$. Note that in this situation
$P_1$ maps the rays of $\Sigma$ onto the rays of $\Delta. $

\begin{lemma}\label{lem:imagesOfRays}
Let $X(A,P, \Sigma) \subseteq Z$ be an explicit arrangement variety with maximal orbit quotient $\PP_c \subseteq \PP_r$ as in Construction~\ref{constr:MOQ}.
Then any ray of $\Sigma \sqcap \trop(X)$
is either projected onto the origin or onto a ray
of $\Delta \sqcap \trop(\PP_c)$.
\end{lemma}
\begin{proof}
Let $\varrho$ be any ray of $\Sigma \sqcap \trop(X)$.
Then there exist cones $\sigma \in \Sigma$ and $\tau \in \trop(X)$ such that $\varrho = \sigma \cap \tau$. 
By construction of the quasifan structure on $\trop(X)$ 
we have 
$\tau = P_1(\tau) \times \QQ^s$ with a cone $P_1(\tau) \in \trop(Y)$. Therefore, we have
$$
P_1(\varrho) = P_1(\sigma \cap \tau) = P_1(\sigma) \cap P_1(\tau) = \cup_{\delta \in D} (\delta \cap P_1(\tau)),
$$
where $D$ is a subset of $\Delta$ and the last equality follows as $\Delta$ is complete.
As $P_1(\varrho)$ is of dimension at most one, we conclude that there exists a $\delta \in D$ with
$P_1(\varrho) = \delta \cap P_1(\tau)$ and the assertion follows.
\end{proof}

Let $B$ denote the matrix whose columns are the primitive ray generators of $\PP_r$. Then we have $\PP_c = X(A,B,\Delta) \subseteq  \PP_r$ and we may use the notions of Definitions~\ref{def:Pcone} and~\ref{def:elemCones}.

\begin{lemma}\label{lem:raysOfElementaries}
Let $\sigma\in\Sigma$ be any special cone.
Then the following statements hold:
\begin{enumerate}
\item 
If $\delta:= P_1(\sigma)$
is an elementary $B$-cone, then we have $P_1(\varrho) = \varrho_{\delta}$ for all $\varrho\in(\sigma\sqcap\trop(X))^{(1)}\setminus\sigma^{(1)}$.
\item
If $\sigma$ is elementary
and $\varrho$ is a ray in $\sigma \sqcap \trop(X)$ with $\varrho \cap \sigma^\circ \neq \emptyset$. Then
$\varrho = \varrho_\sigma$ holds.
\end{enumerate}
\end{lemma}
\begin{proof}
We prove (i).
Let $\varrho \in \sigma \sqcap \trop(X)$
be any ray with $\varrho \cap \sigma^\circ \neq \emptyset.$ Then there exists $\tau \in \mathrm{trop}(X)$ such that $\varrho = \sigma \cap \tau$ holds.
As $\sigma$ is special we have $\varrho \not \subseteq \lambda_{\mathrm{lin}}$.
Therefore, applying Lemma~\ref{lem:imagesOfRays} yields 
$$
P_1(\sigma \cap \tau) = P_1(\varrho) \in (\Delta')^{(1)}.
$$
Using 
$\emptyset \neq \varrho \cap \sigma^\circ  = \varrho \setminus \left\{0\right\}$
we conclude
$P_1(\sigma)^\circ \cap P_1(\varrho) \neq \emptyset$.
As the Bergman fan structure fixed on $\trop(Y)$ coarsens the matroid fan structure 
defined in Remark~\ref{rem:tropical} 
we obtain that
for every cone of $\delta \in \Delta$ 
there exists at most one ray 
$\varrho' \in \Delta'$ such that
$\varrho' \cap \delta^\circ \neq \emptyset$
and this ray equals $\varrho_{\delta}$.

We prove (ii).
As $\sigma$ is elementary, the projection $P_1(\sigma)$ is an
elementary $B$-cone.
As $\varrho \cap \sigma^\circ \neq \emptyset$ we conclude $\varrho \not \in \sigma^{(1)}$. 
Therefore we may apply (i) and obtain $P_1(\varrho) = \varrho_\delta.$
Due to the structure of $\sigma$ there is exactly
one ray in $\varrho_\delta \times \QQ^s$
and this ray equals $\varrho_\sigma$. This completes the proof.
\end{proof}

\begin{lemma}\label{lem:elemetaryreicht}
Let $\sigma \in \Sigma$ be any special or big cone.
If $\varrho_{\sigma_1}=\varrho_{\sigma_2}$ holds for any two elementary $P$-cones $\sigma_1,\sigma_2\subseteq\sigma$, then $\sigma$ is elementary.
\end{lemma}
\begin{proof}
Assume $\sigma$ is not elementary and denote by $I$ the set of indices $i$ such that there exists at least one index $1 \leq k \leq n_i$ with $\cone(v_{ik}) \in \sigma^{(1)}.$ Then there exists $t \in I$ and cones
$$
\tau
\ = \
\cone(v_{ij_i}; \  i \in I)
\ \subseteq \
\sigma_0,
\qquad
\tau'
\ = \
\cone(v_{ij'_i}; \  i \in I)
\ \subseteq \
\sigma_0
$$
with $j_t \ne j'_t$ and
$j_i = j'_i$ for all $i \neq t$. 
In particular we have $\tau \neq \tau'$.
Consider $v_{\tau}$ and $v_{\tau'}$ and 
denote by $c_\tau$ and $c_{\tau'}$ the
respective greatest common divisors of their entries.
Here, we may assume that
$c_{\tau}^{-1}l_{tj_t} \ge c_{\tau'}^{-1}l_{tj_t'}$ holds.
Moreover, as $\varrho_\tau = \varrho_{\tau'}$ holds, we have
$c_\tau^{-1} v_\tau = c_{\tau'}^{-1}v_{\tau'}$.
We conclude
$$
c_{\tau'}^{-1}\ell_{\tau',t} v_{tj_t'}
\ = \
c_{\tau}^{-1}\ell_{\tau,k} v_{tj_t}
+
\sum_{i \in I, i \ne t}
(c_{\tau}^{-1}\ell_{\tau,i} - c_{\tau'}^{-1}\ell_{\tau',i}) v_{ij_i}
$$
and
$c_{\tau}^{-1}\ell_{\tau,i} \ge c_{\tau'}^{-1}\ell_{\tau',i}$ holds for all $1 \le i \le r$.
This implies $v_{tj_t'} \in \tau$. 
But as $\cone(v_{tj_t'})$
is an extremal ray of $\sigma_0$ and
$\tau' \subseteq \sigma_0$ holds,
$\cone(v_{tj_t'})$ is also an extremal ray of $\tau$.
This contradicts the choice of $j_t'$.
\end{proof}
\begin{proof}[Proof of Proposition~\ref{prop:disc}]
We show "$\subseteq$". Let $\varrho$ be any ray of $\Sigma \sqcap \trop(X)$. 
Then there exist 
$\sigma \in \Sigma$ and $\lambda \in \trop(X)$ with
$\sigma \cap \lambda = \varrho$ and we will always assume $\sigma$ and $\lambda$ to be minimal with this property.
Note that if $\sigma$ is a leaf cone, we have $\sigma \subseteq |\trop(X)|$
and due to the
quasifan structure fixed on $\trop(X)$ we obtain $\sigma = \varrho$ for a ray $\varrho$ of $\Sigma$.
So assume $\sigma$ is not a leaf cone.
We distinguish between the following two cases:

\emph{Case 1:} We have $\lambda = \lambda_{\mathrm{lin}}$
and with $\sigma \cap \lambda \neq \emptyset$ we conclude that $\sigma$ is big.
In particular, there exists a $P$-elementary cone $\sigma_1 \subseteq \sigma$ 
with 
$$
\varrho_{\sigma_1} = \sigma_1 \cap \lambda_{\mathrm{lin}} = \sigma \cap \lambda_{\mathrm{lin}} = \varrho.
$$
As this equality holds for any $P$-elementary cone $\sigma_1 \subseteq \sigma$
we can apply Lemma~\ref{lem:elemetaryreicht}
and conclude $\sigma$ is $P$-elementary and $\varrho = \varrho_\sigma$.

\emph{Case 2:} We have 
$\lambda \neq \lambda_\mathrm{lin}.$
Due to minimality of $\sigma$ and $\lambda$ we 
have $\varrho \not \subseteq \lambda_\mathrm{lin}$ which implies that $\sigma$ is 
special and
$\sigma^{\circ} \cap \varrho = \sigma^{\circ} \cap \lambda \neq \emptyset$. In particular, either $\sigma$ 
fulfills already condition (i) of Definition~\ref{def:elemCones} and is elementary, or it contains a $P$-elementary cone $\sigma_1$. 
In the latter case we conclude
$$
\varrho_{\sigma_1} = \sigma_1 \cap \lambda = \sigma \cap \lambda = \varrho.
$$
As the above equality does not depend on the choice of $\sigma_1 \subseteq \sigma$ we conclude that $\sigma$ is elementary  due to Lemma~\ref{lem:elemetaryreicht} and $\varrho = \varrho_{\sigma}$. 

We show "$\supseteq$". By construction the rays of $\Sigma$ are supported on the tropical variety. Therefore it is only left to show that any ray 
$\varrho_\sigma$ lies in $(\Sigma \sqcap \trop(X))^{(1)}$. This follows by definition of elementary and Lemma~\ref{lem:raysOfElementaries}~(ii).
\end{proof}

Due to Proposition~\ref{prop:CoxRing} the Cox ring $R(A,P)$ of an explicit special arrangement variety $X:=X(A,P,\Sigma)\subseteq Z$ is a complete intersection ring and therefore so is the Cox ring of any explicit weakly tropical resolution $X'$.
Applying~\cite{ArDeHaLa2015}*{Prop. 3.3.3.2} 
we obtain the canonical class of $X'$ 
via the following formula:
\begin{equation}
\KKK_{X'} = - \sum_{\varrho \in (\Sigma')^{(1)}}\deg(T_\varrho)
+ \sum_{i= 1}^{r-c} \deg(g_i') 
\in\Cl(X)\cong\ZZ^{n+m} / \im((P')^*).
\end{equation}

In order to prove Proposition~\ref{prop:disc}
we directly import the notion of a toric canonical $\varphi$-family from~\cite{HiWr2018}.

\begin{definition}
\label{def:toricCanonPhiFamily}
Let $X \subseteq Z$ be 
a $\QQ$-Gorenstein explicit arrangement variety
admitting an semi-locally toric
explicit weakly tropical resolution
$\varphi \colon Z' \to Z$. 
A \emph{toric canonical $\varphi$-family} 
is a family $(Z'_{\sigma'},D_{\sigma'})_{\sigma' \in \Sigma'}$, 
where  the $Z'_{\sigma'} \subseteq Z'$ are the affine charts of $Z'$ and the $D_{\sigma'}$ are $T'$-invariant 
Weil divisors on $Z'$ such that for every $\sigma' \in \Sigma'$ 
the following holds:
\begin{enumerate}
\item
$D_{\sigma'} \vert_{X'}$ is a canonical divisor on $X'$,
\item
on $Z'_{\sigma'}$ we have $D_{\sigma'} = k_{Z'}$,
where $k_{Z'} = - \sum_{\varrho' \in \Sigma'} D_{Z'}^{\varrho'}$.
\end{enumerate}
\end{definition}

Let $u \in M_\QQ$ be a rational character. Then the multiplicity of $\div(\chi^{u})$
along the divisor $D_Z^\varrho$ corresponding to a ray 
$\varrho \in \Sigma$ is given as $\langle u, v_\varrho \rangle$,
where, $v_\varrho \in \varrho$ denotes 
the primitive lattice vector inside~$\varrho$.
Now, in the situation of Definition~\ref{def:toricCanonPhiFamily}
let $(Z'_{\sigma'},D_{\sigma'})_{\sigma' \in \Sigma'}$
be a toric canonical $\varphi$-family. 
Choose for every $\sigma' \in \Sigma'$ a cone $\sigma \in \Sigma$ with 
$\sigma' \subseteq \sigma$
and a $u_{\sigma'} \in M_\QQ$ with 
$\varphi_* D_{\sigma'} = \div(\chi^{u_{\sigma'}})$ 
on $Z_\sigma \subseteq Z$.
Then we have
$$ 
\mathcal{A} \cap \sigma'
\ = \ 
\sigma' \cap \{v \in N_\QQ; \ \langle u_{\sigma'} , v \rangle \ge -1 \}.
$$

\begin{proof}[Proof of Proposition~\ref{prop:disc}]
We prove (i).
By assumption $X' \subseteq Z'$
is semi-locally toric and by~\cite{HiWr2018}*{Lem. 3.16}
there exists a toric canonical $\varphi$-family. 
Therefore, explicitly constructing a pair 
$(Z'_{\varrho_{\sigma}}, D_{\varrho_\sigma})$ as in Definition
\ref{def:toricCanonPhiFamily} we can 
calculate the discrepancy along 
$D_{X'}^{\varrho_\sigma}$:

Consider the ray $\varrho_{\sigma} \in \Sigma'$
and let $g_1', \ldots, g_r'$ be the defining relations
of $\mathcal{R}(X')$.
Then in each $g_t'$ we can choose a monomial not divisible by the
variable $T_{\varrho_\sigma}$.
Let us denote this monomial with 
$T_t^{l_t}= T_{t1}^{l_{t1}}\cdots T_{tn_t}^{l_{tn_t}}$.
Then we may choose
$$D_{\varrho_\sigma} := \sum_{t =1}^r
\sum_{j = 1}^{n_t}l_{tj} D_{\varrho'_{tj}} - \sum_{\varrho' \in (\Sigma')^{(1)}}D_{\varrho'},$$
where $D_{\varrho'_{tj}}$ denotes the divisor corresponding to the variable
$T_{tj}$ in $\mathcal{R}(X')$.
As $X'$ has complete intersection Cox ring, the pullback $D_{\varrho_\sigma}|_{X'}$ is a canonical divisor on $X'$. Moreover, the push forward $\varphi_*(D_{\varrho_{\sigma}})$ is $\QQ$-Cartier and by construction we have $D_{\varrho_\sigma} = k_{Z'}$ on $Z'_{\varrho_{\sigma}}$.
In particular, we have constructed a tuple $(Z'_{\varrho_{\sigma}}, D_{\varrho_\sigma})$ as claimed.
Now, let $u \in \QQ^{r+s}$ be an element such that $\mathrm{div}(\chi^{u}) = \varphi_*(D_{\varrho_\sigma})$ holds on $Z_{\sigma}$. Then due to \cite{HiWr2018}*{Rem. 3.11} we have
$$
\mathrm{discr}_X(D_{X'}^{\varrho_\sigma}) = -1 - \bangle{u, v_{\varrho_\sigma}}. 
$$
Therefore, using $v_{\sigma} = v_{\varrho_\sigma} \cdot c_{\sigma}$, we obtain the assertion with
$$
\bangle{u, v_{\sigma}} 
=
\bangle{u, \sum_{i=0}^r \ell_{\sigma, i}v_{ij_i}}
= 
\sum_{i=0}^r\ell_{\sigma,i}\bangle{u, v_{ij_i}}
=
- \ell_\sigma.
$$
Here the last equality holds, as 
$D_{ij_i}$ occurs in $\varphi_*(D_{\varrho_\sigma})$
if and only if there exists a term $T^{\nu} \cdot T_i^{l_i}$
with $\nu \in \ZZ_{\geq 0}^{r+s}$
which is not divisible by $T_{\sigma_\varrho}$
in one of the shifts.
Using (i) assertion (ii) 
follows from the definition of the anticanonical complex.
\end{proof}

\section{Proof of Theorem~\ref{thm:classification}}\label{sec:7}
This section is dedicated to the proof of Theorem~\ref{thm:classification}.
We start by investigating the critical values of the maximal orbit quotient for honestly special arrangement varieties.

\begin{proposition}\label{prop:fiveLines}
Let $X$ be a projective honestly special arrangement
variety of complexity two. Then the maximal orbit quotient
$X\dashrightarrow \PP_2$ has at least five lines as its critical values.
\end{proposition}
\begin{proof}
We realize $X$ as an explicit $\TT$-variety $X(\alpha,P,\Sigma)$ as in~\cite{HaHiWr2019}. For one up to three lines, the Cox ring $R(\alpha,P)$ is a polynomial ring and thus $X$ is toric.
In case of four lines, $X$ admits either a torus action of complexity one, having one hidden free variable as in Example~\ref{ex:notHonest} or is a general arrangement variety. Thus the assertion follows.
\end{proof}

\begin{remark}\label{rem:quadratic}
Let $X:=X(A,P,\Sigma)\subseteq Z$ be an explicit honestly special arrangement variety of complexity two and assume the maximal orbit quotient $X\dashrightarrow \PP_2$ has a line arrangment of five lines as its critical values. 
Then we obtain the following two types of relations, as the remaining cases are either general or of complexity one:
$$
\begin{array}{rll}
\text{(I)}&g_1 = T_0^{l_0}+T_1^{l_1}+T_2^{l_2}+T_3^{l_3},&g_2 = T_1^{l_1}+aT_2^{l_2}+T_4^{l_4}
\\
\text{(II)}& g_1=T_0^{l_0}+T_1^{l_1}+T_3^{l_3}, & g_2=T_0^{l_0}+T_2^{l_2}+T_4^{l_4}
\end{array}
$$
Now specializing to dimension three and divisor class group of rank at most two, we obtain $n_i\leq 2$, $n_i = 2$ for at most two indices $i=0,\ldots,4$ and $m\leq 2$. Assuming that $X$ is of finite isotropy order at most two, we obatin $l_i = 2$ for $n_i = 1$ and $l_i = (1,1)$ for $n_i = 2$. In particular the polynomials $g_1,g_2$ are quadratic. 

Note that in question of isomorphy, the distribution of the $n_i$ is important: Two rings $R(A,P)$ and $R(A',P')$ of the same type (I) or (II), where the vector $(n_0,\ldots,n_r)$ is a permutation of the vector $(n_0',\ldots,n_r')$ do not need to be isomorphic, see i.a.\ Nos. 7 and 9 in our list.
\end{remark}

\begin{remark}
Let $X(A,P,\Sigma)\subseteq Z$ be an explicit arrangement variety with divisor class group of rank two
and consider the weight matrix $Q$ whose columns consist 
of the free part of the weights $w_{ij}:=\deg(T_{ij})$ resp. $w_k:=\deg(T_k)$.
We write a $\QQ$-basis for the kernel of $P$ in the 
rows of a  matrix $\tilde{Q}$
and define a vector $-w_X$:
$$
\tilde Q= [\tilde w_{01},\ldots,\tilde w_{rn_r}],\qquad
-w_X:=\sum \tilde w_{ij} - (r-c) \sum_j l_{0j}\tilde w_{0j}.
$$
Then there is a $\QQ$-linear isomorphism mapping the columns of $Q$ on the columns of $\tilde Q$ and thus the canonical class $-\KKK_x$ on $-w_X$. This isomorphism is either orientation preserving or reversing. Therefore, in question of the position of weights inside $\Eff(X)$ as in Remark~\ref{rem:divClass}, it suffices to look at {\em rational weight matrices} $\tilde{Q}$ with {\em rational anticanonical vectors} $-w_X$.
\end{remark}

\begin{remark}
The following list of admissible operations on $P$ do not effect the isomorphy type of the rings $R(A,P)$:
\begin{enumerate}
    \item Swapping two columns inside a block $v_{ij_1},\ldots, v_{ij_{n_i}}$.
    \item Adding multiples of the upper $r$ rows to one of the last $s$ rows.
    \item Any elementary row operations among the last $s$ rows.
    \item Swapping two columns inside the last $s$ rows lying under the $0$-block of the matrix $P_0$.
\end{enumerate}
\end{remark}

\begin{proof}[Proof of Theorem~\ref{thm:classification}]
Let $X:=X(A,P,\Sigma)$ be an explicit special arrangement variety as in Theorem~\ref{thm:classification}. 
Then the Cox ring of $X$ is given as a ring $R(A,P)$ as in Remark~\ref{rem:quadratic}.
In a first step, we bound the entries of $P$ to obtain a list of candidates. Note that our computations are independent of the type of the relations and the distribution of the $n_i$.

\vspace{10pt}
\noindent
\emph{Case $\rk(\Cl(X)) = 1$:}
Due to Remark~\ref{rem:quadratic} we are left with the following three cases:
\begin{enumerate} 
\item[(a)] $n= 5$ and $m=1$
\item[(b)] $n= 6$ and $m=0$
\end{enumerate}
Note that in case of divisor class group of rank one, every $\bar{X}$-face is an $X$-face. 

\vspace{2pt}
\noindent
\emph{Case (a):}
After applying suitable admissible operations on $P$ we may assume that we are in the following situation
$$
P=\left[
\begin{array}{ccccccc}
-2&2&0&0&0&0\\
-2&0&2&0&0&0\\
-2&0&0&2&0&0\\
-2&0&0&0&2&0\\
x&1&1&1&1&1
\end{array}\right].
$$
Now the big cone $\sigma$ gives rise to a vertex
$v_\sigma'$ of the anticanonical complex
$$
\sigma = \mathrm{cone}(v_{01},v_{11},v_{21},v_{31},v_{4,1}),
\qquad
v_{\sigma}' = [0,0,0,0,4+x].
$$
Thus $x = -5$ holds due to the singularity type of $X$.

\vspace{2pt}
\noindent
\emph{Case (b):}
After applying suitable admissible operations on $P$ we may assume that we are in the following situation
$$
P=\left[
\begin{array}{ccccccc}
-1&-1&2&0&0&0\\
-1&-1&0&2&0&0\\
-1&-1&0&0&2&0\\
-1&-1&0&0&0&2\\
x&y&1&1&1&1
\end{array}\right],
$$
where we may assume $x > y$. Now the two big cones
$$
\sigma_1=\mathrm{cone}(v_{01},v_{11},v_{21},v_{31},v_{41}),
\qquad
\sigma_2=\mathrm{cone}(v_{02},v_{12},v_{21},v_{31},v_{41})
$$
give the vertices $v_{\sigma_1}'$ and $v_{\sigma_2}'$ of the anticanonical complex 
$$
v_{\sigma_1}'=[0,0,0,0,2+x],
\qquad
v_{\sigma_2}'=[0,0,0,0,2+y].$$
We conclude $x = -1$ and $y = -3$ due to the singularity type of $X$.

\vspace{10pt}
\noindent
\emph{Case $\rk(\Cl(X)) = 2$:}
Due to Remark~\ref{rem:quadratic} we are left with the following three cases:

\begin{enumerate} 
\item[(a)] $n= 5$ and $m=2$
\item[(b)] $n= 6$ and $m=1$
\item[(c)] $n= 7$ and $m=0$
\end{enumerate}

\vspace{2pt}
\noindent
\emph{Case (a):}
After applying suitable admissible operations on $P$ we may assume that we are in the following situation
$$
P=\left[
\begin{array}{ccccccc}
-2&2&0&0&0&0&0\\
-2&0&2&0&0&0&0\\
-2&0&0&2&0&0&0\\
-2&0&0&0&2&0&0\\
x&1&1&1&1&1&-1
\end{array}\right]
$$
Now a rational weight matrix $\tilde Q$ and the corresponding rational vector $-w_X$ is given as 
$$
\tilde Q=\left[
\begin{array}{ccccccc}
1&1&1&1&1&-x-4&0\\
0&0&0&0&0&1&1
\end{array}
\right],\qquad
-w_X = \left[
\begin{array}{c}
     -x-3\\
     2
\end{array}
\right]
$$
In case $-x-4 \leq 0$ we have $\SAmple(X)=\cone([1,0],[0,1])$
and thus $-x-3 > 0$ due to the Fano property of $X$.
This implies $-4\leq x < -3$ and thus $x=-4$; a contradiction to the primality of the columns of $P$.
In case $-x-4 > 0$, we have 
$\SAmple(X) = \cone([1,0],[-x-4,1])$.
This implies
$$-w_X=\left[\begin{array}{c}
-x-3\\
2
\end{array}\right] = (x+5)
\left[\begin{array}{c}
1\\
0
\end{array}\right]
+2
\left[\begin{array}{c}
-x-4\\
1
\end{array}\right]$$
with $x+5>0$ due to the Fano property of $X$.
Thus, we obtain $-5<x<-4$; a contradiction.

\vspace{2pt}
\noindent
\emph{Case (b):}
After applying suitable admissible operations on $P$ we may assume that we are in the following situation
$$
P=\left[
\begin{array}{ccccccc}
-2&1&1&0&0&0&0\\
-2&0&0&2&0&0&0\\
-2&0&0&0&2&0&0\\
-2&0&0&0&0&2&0\\
1&x&y&1&1&1&1
\end{array}\right],
$$
where we may assume $x>y$.
Due to completeness of $X$, we obtain an 
elementary big cone $\sigma\in\Sigma$ defining a 
vertex $v_{\sigma}'$ of the anticanonical complex:
$$\sigma=\mathrm{cone}(v_{01},v_{12},v_{21},v_{31},v_{41}),
\qquad
v_\sigma' = [0,0,0,0,y+2].$$
Thus, we conclude $y=-3$ due to the singularity type of $X$.
Now a rational weight matrix $\tilde Q$ and the corresponding  rational vector $-w_X$ is given as
$$
\tilde Q=\left[
\begin{array}{cccccccc}
     1&2&0&1&1&1&-2x-4\\
     1&0&2&1&1&1&2
     \end{array}
\right],\qquad
-w_X=
\left[
\begin{array}{c}
-2x-4\\
2
\end{array}
\right]
$$
In particular, we obtain $\SAmple(X)\subseteq \QQ^2_{\geq 0}$. As $X$ is Fano, this implies $-2x-2\geq 0$ and thus $x\leq -1$.

\vspace{2pt}
\noindent
\emph{Case (c):}
After applying suitable admissible operations on $P$ we may assume that we are in the following situation
$$
P=\left[
\begin{array}{ccccccc}
-2&1&1&0&0&0&0\\
-2&0&0&1&1&0&0\\
-2&0&0&0&0&2&0\\
-2&0&0&0&0&0&2\\
1&x&y&z&0&1&1
\end{array}\right]
$$
where we may assume $x>y$ and $z > 0$.
Now, due to completeness of $X$ we obtain two elementary big cones 
$$\sigma_1 = \mathrm{cone}(v_{01},v_{11},v_{21},v_{31},v_{41}),
\qquad
\sigma_2 = \mathrm{cone}(v_{01},v_{12},v_{22},v_{31},v_{41})
$$
defining the following vertices of the anticanonical complex
$$v_{\sigma_1}' = [0,0,0,0,1+(2/3) x + (2/3) z],
\qquad
v_{\sigma_2}' = [0,0,0,0,1+(2/3)y].$$
Thus we conclude 
$$-3\leq y \leq -2 \qquad \text{and} \qquad 0< 1+(2/3)x+(2/3)z \leq 1$$
due to the singularity type of $X$. This implies 
$-2\leq x < 0$ and thus $0 < z \leq 2$.

Now, any of the configurations above gives a ring $R(A,P)$.
A direct computation shows that in all cases, the generators $T_{ij}$ are $K$-prime. To obtain our list, we computed the anticanonical complexes for all configurations and checked for canonicity using the characterization at the beginning of Section~\ref{sec:6}. After removing some redundancy, we obtain the varieties in our list.
\end{proof}

\section{Beyond arrangement varieties}\label{sec:8}
In this section we go one step beyond arrangement varieties and consider varieties whose maximal orbit quotient decomposes as a product of arrangements. Note that their Cox rings (without the grading) already appeared in Section~\ref{sec:3}: As rings they are isomorphic to decomposable rings $R(A,P_0)$. 
We follow the ideas of Section~\ref{sec:2} by constructing Cox rings of these varieties and realize them as explicit $\TT$-varieties.
As an application we obtain in Proposition~\ref{prop:notTrue} a criterion to determine the true complexity of a special arrangement variety. 
Finally, we give a full classification in the smooth case for projective varieties up to Picard number two and characterize the Fano property of these varieties.

\begin{definition}
An {\em arrangement-product variety} is a variety $X$ with an effective action of an algebraic torus $\TT\times X\rightarrow X$ having $X\dashrightarrow \PP_{c_1}\times \ldots\times\PP_{c_t}$ with $t>1$ as a maximal orbit quotient and the critical values is given as a collection of products 
$$\PP_{c_1}\times \ldots\PP_{c_{i-1}}\times D_k^{(i)}\times\PP_{c_{i+1}}\times \ldots\times \PP_t$$
where $i=1,\ldots, t$ and the collection $D_k^{(i)}$ is a hyperplane arrangement in $\PP_{c_i}$.
\end{definition}

\begin{remark}
Arrangement-product varieties have finitely generated Cox rings due to~\cite{HaSu2010}*{Thm 1.2}.
\end{remark}

We go on by constructing Cox rings of arrangement-product varieties. We will make use of the notion of (in-)decomposability of rings $R(A,P)$ as in Definition~\ref{def:indecomposable}.

\begin{construction}\label{constr:RprodAP}
Consider indecomposable rings $R(A^{(i)},P_0^{(i)})$ for $i=1,\ldots,t$ from Construction~\ref{constr:RAP0} with $m_i'=0$.
Choose integers $s>0$, $m\geq 0$ and set
$$c:=c^{(1)}+\ldots+c^{(t)},\ n:=n^{(1)}+\ldots+n^{(t)}\text{ and } r:=r^{(1)}+\ldots+r^{(t)}.$$
We build up a new $(c+t)\times(r+t)$ resp. $(r+s)\times(n+m)$ matrices
$$A:=\left[
\begin{array}{ccc}
A^{(1)}&&\\
&\ddots&\\
&&A^{(t)}
\end{array}
\right],\qquad
P:=\left[
\begin{array}{c}
P_0\\
\hline
d
\end{array}
\right]
:=
\left[
\begin{array}{cccccc}
P_0^{(1)}&&&0&\ldots&0\\
&\ddots&&\vdots&&\vdots\\
&&P_0^{(t)}&0&\ldots&0\\
\hline
&&d&&
\end{array}
\right],
$$
where we require the columns of $P$ to be pairwise different and primitive, generating $\QQ^{r+s}$ as a vectorspace. 
Denote by $e^{(i)}_{kl}$ resp. $e^{(i)}_k$ the canonical basis vectors of $\QQ^{n+m}$ accordingly to the decomposition $n=n^{(1)}+\ldots+n^{(t)}$ and let $Q_0\colon \ZZ^{n+m}\rightarrow \ZZ^{n+m}/\mathrm{im}(P_0^*):=K_0$
be the projection onto the factor group.
We define a $\CC$-algebra
$$R_\mathrm{prod}(A,P_0):=\bigotimes R(A^{(i)},P_0^{(i)})$$
and endow it with a $K_0$-grading by setting
$$\deg(T^{(i)}_{kl}):=Q(e_{kl}^{(i)}),\qquad \deg(S_k):=Q(e_k).$$
Moreover, by considering the projection 
$Q\colon \ZZ^{n+m}\rightarrow \ZZ^{n+m}/\mathrm{im}(P^*):=K$, we define analogously a $K$-graded $\CC$-algebra $R_\mathrm{prod}(A,P)$.
\end{construction}

\begin{remark}
The $K_0$-graded $\CC$-algebras $R_\mathrm{prod}(A, P_0)$ 
are integral, normal, complete intersection 
rings satisfying 
$$ 
\dim(R_\mathrm{prod}(A,P_0)) \ = \ n+m-r+c,
\qquad 
R_\mathrm{prod}(A,P_0)^* \ = \ \CC^*
$$
and the $K_0$-grading is the finest possible grading on $R_\mathrm{prod}(A,P_0)$, leaving the variables and the relations homogeneous. Moreover, it is effective, pointed, factorial and of complexity~$c$. Considering the rings $R_\mathrm{prod}(A,P)$, the $K$-grading is effective, factorial and of complexity $c$ and, if the columns of $P$ generate $\QQ^{r+s}$ as a cone, it is pointed as well.
\end{remark}

\begin{remark}
Every arrangement-product variety has a $K$-graded $\CC$-algebra $R_\mathrm{prod}(A,P)$ as in Construction~\ref{constr:RprodAP} as its Cox ring.
\end{remark}

As done in Section~\ref{sec:2}, we use the rings $R_\mathrm{prod}(A,P)$ to construct explicit $\TT$-varieties following precisely the same steps as in Construction~\ref{constr:XAPSigma}.
We will denote the resulting explicit $\TT$-varieties with $X_\mathrm{prod}(A,P,\Sigma)\subseteq Z.$

\begin{remark}\label{rem:prodCompl}
Let $X:=X_\mathrm{prod}(A,P,\Sigma)\subseteq Z$ be an explicit arrangement-product variety.
Then the subtorus action of $\TT^s\subseteq\TT^{r+s}$ on $Z$ leaves $X$ invariant. This turns $X$ into a $\TT^s$-variety of complexity $c = c_1+\ldots+c_t$.
\end{remark}

\begin{proposition}\label{prop:notTrue}
Let $X:=X(A,P,\Sigma)\subseteq Z$ be an explicit special arrangement
variety of complexity $c$ with a decomposable ring $R(A,P)$. Then $X$ is not of true complexity~$c$.
\end{proposition}
\begin{proof}
If $R(A,P)$ is decomposable, then $X$ can be regained as an explicit $\TT'$-variety out of its Cox ring $R_\mathrm{prod}(A,P)$, where the $\TT'$-action is of lower complexity by Remark~\ref{rem:prodCompl}.
\end{proof}

We now turn to our main results concerning smoothness of arrangement-product varieties.
We will without further explanation use the language 
of explicit $\TT$-varieties as done in Section~\ref{sec:3}.
In particular, the smoothness criteria from Remark~\ref{rem:smoothCrit} can be applied in our situation.

\begin{proposition}
Let $X$ be a projective arrangement-product variety of Picard number one. Then $X$ is singular.
\end{proposition}
\begin{proof}
By definition the Cox ring $R_\mathrm{prod}(A,P)$ is decomposable into $t>1$ indecomposable rings $R^{(i)}$. 
Therefore the cone
$$\gamma^{(1)} := \cone(e_{kl}^{(1)}; \ 0 \leq k \leq r^{(1)}, 1 \leq l \leq n_k^{(1)})$$ 
is an $\overline{X}$-face whose corresponding $\overline{X}$-stratum is singular.
As $X$ is of Picard number one, any $\overline{X}$-face is an $X$-face and we conclude that $X$ is singular.
\end{proof}

\begin{theorem}\label{thm:someSmoothOnes}
Every smooth projective arrangement-product variety of Picard number two is isomorphic to a variety $X$ specified by its Cox ring 
$$
\mathcal{R}(X)  = \CC[T_{11}, \ldots, T_{1k_1},T_{21}, \ldots T_{2k_2}]/\bangle{g_1, g_2}
$$
where
$$g_i  =  
\begin{cases}
T_{i1}T_{i2} + \ldots + T_{ik_i-1}T_{ik_i}, & k_i \geq 6 \quad  \text{even}\\
T_{i1}T_{i2} + \ldots + T_{ik_i-2}T_{ik_i-1} + T_{ik_i}^2, & k_i \geq 5 \quad \text{odd},
\end{cases}
$$
the matrix $Q$ of generator degrees and an ample class $u\in\Cl(X)=\ZZ^2$
$$
Q=
\left[
\begin{array}{ccc|cccc}
1 & \ldots & 1 & a_1 & a_2 & \ldots & a_{k_2}\\
0 & \ldots & 0 & 1 & 1 & \ldots & 1
\end{array}
\right],
\qquad 
u = [a_1 + 1,1],
$$
where we have $a_i \geq a_{i+2} \geq 0$ and $a_i + a_{i+1} = 0$ for $i$ odd
and $a_{k_2} =  0$ if $k_2$ is odd.
\end{theorem}

\begin{corollary}\label{cor:somesmoothFanoOnes}
A smooth projective arrangement-product variety of Picard number two as in Theorem~\ref{thm:someSmoothOnes} is Fano if and only if
$0 \leq a_1 \leq \frac{k_1 -2}{k_2 - 2}$ holds.
\end{corollary}

\begin{corollary}
If $X$ is a smooth projective arrangement-product variety of Picard number two, then the dimension of $X$ is at least $6$.
\end{corollary}

The rest of this section is dedicated to the proofs of the previous statements.

\begin{lemma}\label{lem:tool2}
Let $X := X_\mathrm{prod}(A,P,\Sigma)$ be a $\QQ$-factorial, quasismooth projective arrangement-product variety of Picard number two with Cox ring
$R_\mathrm{prod}(A, P)$ decomposing into $t > 1$ 
indecomposable rings $R^{(i)}:=R(A^{(i)},P^{(i)})$.
Then the following statements hold: 
\begin{enumerate}
    \item 
    Let $w_{kl}^{(i)}$ denote the weights corresponding to the variables of the ring
    $R^{(i)}$. Then the $w_{kl}^{(i)}$ lie either all in $\tau^-$ or in $\tau^+$.
    \item
    For all $0 \leq \alpha \leq r^{(i)}$ where $1 \leq i \leq t$ the number of variables per term 
    $n_\alpha^{(i)}$ is at most~$2$.
    \item
    We have $m^{(i)} = 0$ for all $1 \leq i \leq t$, and $t = 2$ holds.
    \item
    If $n_\alpha^{(i)} =2$ holds for one index $0 \leq \alpha \leq r^{(i)}$ then 
    the corresponding ring $R^{(i)}$ has exactly one defining relation.
    \item 
    If $n_\alpha^{(1)} = 2 = n_\beta^{(2)}$ holds for two indices $0 \leq \alpha \leq r^{(1)}$ and $0 \leq \beta \leq r^{(2)}$ then we have $$l_{\alpha1}^{(1)} = l_{\alpha2}^{(1)} = l_{\beta1}^{(2)} = l_{\beta2}^{(2)} = 1.$$
\end{enumerate}
\end{lemma}
\begin{proof}
We prove (i). By construction the cone 
$$\gamma^{(i)} := \cone(e_{kl}^{(i)}; \ 0 \leq k \leq r^{(i)}, 1 \leq l \leq n_k^{(i)})$$ 
is an $\overline{X}$-face. As $R_\mathrm{prod}(A,P)$ is decomposable with $t >1$ the corresponding $\overline{X}$-stratum is singular. 
Now assume that not all weights $w_{kl}^{(i)}$ lie in the same cone $\tau^+$ or $\tau^-$. Then $\gamma^{(i)}$ is $X$-relevant which contradicts quasismoothness. 

We turn to (ii). Assume there exists a ring $R^{(i)}$
such that $n_\alpha^{(i)} > 2$ holds for one index $0 \leq \alpha \leq r^{(i)}$. 
Due to (i) we may assume that all weights $w_{kl}^{(i)}$
lie in $\tau^-$. As there have to be at least two weights in $\tau^+$
there exists an index $1 \leq j \leq t$ such that
all weights $w_{kl}^{(j)}$ lie in $\tau^+$.
We obtain an $X$-face
$\cone(e_{\alpha1}^{(i)}, \gamma^{(j)})$ whose corresponding $\overline{X}$-stratum is singular as $n_\alpha^{(i)} > 2$ holds; a contradiction.

We prove (iii).
Assume $m^{(j)} > 0$ holds for at least one $1 \leq j \leq t$. 
Suitably renumbering we may assume $j =1$.
Moreover, with the same arguments as in the proof of Lemma~\ref{lem:tool}~(i)
we may assume that $w_k^{(1)} \in \tau^+$ holds for all $1 \leq k \leq m^{(1)}$. 
Due to Remark~\ref{rem:normForm} there are at least two weights
that lie in $\tau^-$ and using (i) we conclude that there exists
$1 \leq i \leq t$ such that all weights $w_{kl}^{(i)}$ lie in $\tau^-$.
This gives an $X$-face 
$\cone(\gamma^{(i)}, e_k^{(1)})$
with singular stratum as $t> 1$ holds; a contradiction to quasismoothness.

Now assume $t \geq 3$ holds. As each of $\tau^+$ and $\tau^-$ have to contain at least two weights we may assume
that all weights $w_{kl}^{(1)}$ lie in $\tau^-$ and 
all weights $w_{kl}^{(2)}$ lie in $\tau^+$.
This gives an $X$-face 
$
\cone(\gamma^{(1)}, \gamma^{(2)})
$
which is singular as $t \geq 3$ holds. This contradicts quasismoothness.

We turn to (iv).
By renumbering we may assume $i=1$.
Let $n_\alpha^{(1)} = 2$.
Then $\cone(e_{\alpha1}^{(1)}, \gamma^{(2)})$ is an $X$-face whose stratum is singular if there is another defining relation in $R^{(1)}$. This proves the assertion.

We prove (v). Due to (iii) and Remark~\ref{rem:normForm}
we may assume that all weights $w_{kl}^{(1)}$
lie in $\tau^+$ and all weights $w_{kl}^{(2)}$ lie in $\tau^-$.
This implies that the cones
$\cone(e_{\alpha l}^{(1)}, e_{\beta l'}^{(2)})$
with $l,l' \in \left\{1,2\right\}$
are $X$-faces.
As the corresponding $\overline{X}$-strata
have to be smooth, the assertion follows. 
\end{proof}

\begin{proof}[Proof of Theorem~\ref{thm:someSmoothOnes}]
By Construction $R_\mathrm{prod}(A,P)$ admits a decomposition into
indecomposable rings $R^{(i)}:=R(A^{(i)},P^{(i)})$, where $1 \leq i \leq t$ and $t>1$ holds. 
Applying Lemma~\ref{lem:tool2}
we obtain $t=2$
and $m=0$. Moreover, due to Lemma~\ref{lem:tool2}~(i) we may assume that all weights $w_{kl}^{(1)}$ of the variables of the ring $R^{(1)}$ lie in $\tau^-$ and all weights $w_{kl}^{(2)}$ of the variables of the ring $R^{(2)}$ lie in $\tau^+$.
Since  $X$ is projective and $m = 0$ holds, $\Sigma$ contains 
at least one big cone $\sigma = P(\gamma^*)$ with an $X$-face $\gamma$. Recall that in our situation an $\overline{X}$-face
is an $X$-face if it contains at least one 
ray corresponding to a variable of $R^{(1)}$ and another one corresponding to a variable of $R^{(2)}$. We conclude that there exists $0 \leq \alpha \leq r^{(1)}$ and $0 \leq \beta \leq r^{(2)}$ with
$n_\alpha^{(1)} = 2$ and $n_\beta^{(2)} = 2$, where equality holds due to
Lemma~\ref{lem:tool2}~(ii).
Moreover, Lemma~\ref{lem:tool2}~(v)
implies
$$
l_{\alpha1}^{(1)} = l_{\alpha2}^{(1)} = l_{\beta1}^{(2)} = l_{\beta2}^{(2)} = 1
$$
for all 
$0 \leq \alpha \leq r^{(1)}$ and $0 \leq \beta \leq r^{(2)}$ with
$n_\alpha^{(1)} = 2$ and $n_\beta^{(2)} = 2$.
Applying Lemma~\ref{lem:tool2}~(iv) we are left with one homogeneous defining relation $g_1$ for $R^{(1)}$ and another homogeneous defining relation $g_2$ for $R^{(2)}$
whose weights $\deg(g_1) = w^{(1)}$ and $\deg(g_2) = w^{(2)}$ lie in $\tau^-$ and $\tau^+$ respectively.
In particular, by a suitable unimodular coordinate change on 
$\ZZ^2$ we can achieve 
$w^{(1)} = (w_1,0)$ and $w^{(2)} = (0,w_2)$
with positive integers $w_1, w_2$.
We conclude that for any $\alpha$ and $\beta$
with
$n_\alpha^{(1)} = 2$ and $n_\beta^{(2)} = 2$
as above there exists integers $a$ and $b$ such that
$$
w_{\alpha1}^{(1)} = (w_1/2,b), 
\quad
w_{\alpha2}^{(1)} = (w_1/2,-b), 
\quad
w_{\beta1}^{(2)} = (a,  w_2/2).
\quad
w_{\beta1}^{(2)} = (-a,  w_2/2).
$$
As the cones $\cone(e_{\alpha1}^{(1)}, e_{\beta1}^{(2)})$ and $\cone(e_{\alpha2}^{(1)},e_{\beta1}^{(2)})$ are $X$-faces, we conclude 
\begin{equation}\label{equ:detrechnung}
\frac{1}{4}w_1w_2 +ab.
=
\det(w_{\alpha1}^{(1)},w_{\beta1}^{(2)}) = 1 = \det(w_{\alpha2}^{(1)} ,w_{\beta1}^{(2)}) = \frac{1}{4}w_1w_2 +ab.
\end{equation}
This implies $a = 0$ or $b = 0$ and we may assume the latter holds.
Moreover, we obtain $w_1 = w_2 =2$ and homogeneity of the relations implies that
the relations $g_1$ and $g_2$ are quadratic. 
As equation (\ref{equ:detrechnung}) holds for any choice of $\alpha$
and $\beta$ we are left with the following configuration of weights:
$$
Q=
\left[
\begin{array}{ccc|cccc}
1 & \ldots & 1 & a_1 & a_2 & \ldots & a_{k_2}\\
0 & \ldots & 0 & 1 & 1 & \ldots & 1
\end{array}
\right],
$$
We show that there is
at most one term in each of $g_1$ and $g_2$ with only one variable.
Assume there is more than one. Then the divisor class group contains torsion, see~\cite{HaHiWr2019}*{Prop. 7.3}. This is a contradiction, as $\cone(e_{\alpha1}^{(1)}, e_{\beta1}^{(2)})$ is an $X$-face and therefore the divisor class group is isomorphic to $\ZZ^2$.
Now, the conditions on the $a_i$ follow due to homogeneity of the relations and by suitably renumbering. 
In order to complete the proof it is only left to show, that the varieties in this class are indeed smooth.
This follows directly by checking the criterion of Remark~\ref{rem:smoothCrit}.
\end{proof}

\begin{proof}[Proof of Corollary~\ref{cor:somesmoothFanoOnes}]
In order to prove the statement we consider the varieties 
of Theorem~\ref{thm:someSmoothOnes} and check under which condition the anticanonical class lies in the ample cone.
As $\mathcal{R}(X)$ is a complete intersection ring the anticanonical class is given as
$-\mathcal{K}_X = (k_1 - 2, k_2 - 2)$. 
Moreover, due to Lemma~\ref{lem:tool2} we have $\tau^- = \cone((1,0))$ and
$\tau^{+} = \cone((a_1,1), (-a_1,1))$.
We conclude that
$- \mathcal{K}_X$ lies in the ample cone if and only if
$0 \leq a_1 \leq \frac{k_1 -2}{k_2 - 2}$ holds. 
\end{proof}

\bibliographystyle{unsrt}
\bibliography{bibliography}

\end{document}